\documentclass{amsproc}

\usepackage[cmtip,all]{xy}

\usepackage{etex}
\usepackage{amssymb,euscript,mathtools}
\usepackage[only,mapsfrom]{stmaryrd}

\newcommand*{\Ceu}{\ensuremath{\EuScript{C}}}
\newcommand*{\Deu}{\ensuremath{\EuScript{D}}}
\newcommand*{\Eeu}{\ensuremath{\EuScript{E}}}
\newcommand*{\Heu}{\ensuremath{\EuScript{H}}}
\newcommand*{\Ieu}{\ensuremath{\EuScript{I}}}
\newcommand*{\Jeu}{\ensuremath{\EuScript{J}}}
\newcommand*{\Meu}{\ensuremath{\EuScript{M}}}
\newcommand*{\Neu}{\ensuremath{\EuScript{N}}}
\newcommand*{\Lbb}{\ensuremath{\mathbb{L}}}
\newcommand*{\Rbb}{\ensuremath{\mathbb{R}}}
\newcommand*{\Wcal}{\ensuremath{\mathcal{W}}}
\newcommand*{\Deltabf}{{\bf \Delta}}
\newcommand*{\colim}{\mathop{\rm colim}}
\newcommand*{\hocolim}{\mathop{\rm hocolim}}
\newcommand*{\Ob}[1]{#1}
\newcommand*{\Hom}{\mathop{\rm Hom}\nolimits}
\newcommand*{\Fun}{\mathop{\rm Fun}\nolimits}
\newcommand*{\Nat}{\mathop{\rm Nat}\nolimits}
\newcommand*{\comma}{\mathbin{\downarrow}}
\newcommand*{\Ho}{\mathop{\rm Ho}\nolimits}
\newcommand*{\hyph}{\ifmmode\hbox{-}\else\nobreakdash-\hspace{0pt}\fi}
\newcommand*{\eop}{\leavevmode\unskip\penalty9999 \hbox{}\nobreak\hfill\hbox{\ensuremath{_\Box}}}
\newcommand*{\id}{{\rm id}}
\newcommand*{\ev}{{\rm ev}}
\newcommand*{\pr}{{\rm pr}}
\newcommand*{\Stdin}{{\rm In}}
\newcommand*{\End}{\mathop{\rm End}\nolimits}
\newcommand*{\bang}{{\mathord{!}}}
\newcommand*{\defas}{\mathrel{\mathrel{\mathop:}=}}
\newcommand*{\rar}{\ensuremath{\Rightarrow}}
\newcommand*{\xto}[2][]{\xrightarrow[#1]{#2}}
\newcommand*{\lon}{\nobreak\mskip6mu plus1mu\mathpunct{}\nonscript\mkern-\thinmuskip{:}\mskip2mu\relax}
\newcommand*{\Sets}{{\bf Sets}}
\newcommand*{\sSets}{{\bf sSets}}
\newcommand*{\Top}{{\bf Top}}
\newcommand*{\Ab}{{\bf Ab}}
\DeclareMathOperator{\Ch}{Ch}

\newcommand*{\mate}{\mathrel{\mathchoice%
  {\vcenter{\hbox{${\scriptscriptstyle\bigcirc{\mkern-9mu}\bigcirc}$}}}
  {\vcenter{\hbox{${\scriptscriptstyle\bigcirc{\mkern-9mu}\bigcirc}$}}}
  {\vcenter{\hbox{${\scriptstyle\circ{\mkern-5mu}\circ}$}}}
  {\vcenter{\hbox{${\scriptscriptstyle\circ{\mkern-5mu}\circ}$}}}
}}

\newcommand*{\adj}[1][]{\def\ArgI{#1}\adjRelayI}
\newcommand*{\adjRelayI}[1][]{\def\ArgII{#1}\adjRelayII}
\newcommand*{\adjRelayII}[3][2.2em]{\ensuremath{\SelectTips{cm}{10}\xymatrix@C=#1@1{*++{#2} \ar@<1ex>[r]^-{\ArgI}^-{}="1" & *++{#3} \ar@<1ex>[l]^-{\ArgII}^-{}="2" \ar@{}"1";"2"|(.3){\hbox{}}="3"|(.7){\hbox{}}="4" \ar@{-|}"3";"4"}}}
\newcommand*{\radj}[1][]{\def\ArgI{#1}\radjRelayI}
\newcommand*{\radjRelayI}[1][]{\def\ArgII{#1}\radjRelayII}
\newcommand*{\radjRelayII}[3][2.2em]{\ensuremath{\SelectTips{cm}{10}\xymatrix@C=#1@1{*++{#2} \ar@<-1ex>[r]_-{\ArgII}^-{}="1" & *++{#3} \ar@<-1ex>[l]_-{\ArgI}^-{}="2" \ar@{}"1";"2"|(.3){\hbox{}}="3"|(.7){\hbox{}}="4" \ar@{-|}"4";"3"}}}

\makeatletter
\newsavebox\tboxa
\newsavebox\tboxb
\newlength\tdima
\newcommand*{\overbar}{\mathpalette\@overbar}
\newcommand*{\@overbar}[2]{%
    \sbox{\tboxa}{$\m@th#1\mathrm{#2}$}%
    \setbox\tboxb\null%
    \ht\tboxb\ht\tboxa%
    \dp\tboxb\dp\tboxa%
    \wd\tboxb\wd\tboxa%
    \sbox{\tboxa}{$\m@th#1{#2}$}%
    \setlength\tdima{\the\wd\tboxa}%
    \addtolength\tdima{-\the\wd\tboxb}%
    \sbox{\tboxb}{$\m@th#1\hskip\tdima\overline{\xusebox{\tboxb}}$}%
    \rlap{\usebox\tboxb}{\usebox\tboxa}}
\newcommand*{\xusebox}[1]{\mathord{{\usebox{#1}}}}
\makeatother
\newcommand*{\ol}[1]{\mkern 1.5mu\overbar{\mkern-1.5mu#1\mkern-1.5mu}\mkern 1.5mu}

\newlength\oldarraycolsep

\newdir{ >}{{}*!/-5pt/@{>}}
\newdir{c}{{}*!/-5pt/@{(}}
\newdir^{c}{{}*!/-5pt/@^{(}}
\newdir_{c}{{}*!/-5pt/@_{(}}

\newtheorem{theorem}{Theorem}[section]
\newtheorem{proposition}[theorem]{Proposition}
\newtheorem{corollary}[theorem]{Corollary}

\theoremstyle{definition}
\newtheorem{definition}[theorem]{Definition}
\newtheorem{convention}[theorem]{Convention}
\newtheorem{nomenclature}[theorem]{Nomenclature}
\newtheorem{example}[theorem]{Example}

\theoremstyle{remark}
\newtheorem{remark}[theorem]{Remark}
\newtheorem{observation}[theorem]{Observation}

\numberwithin{equation}{section}

\begin{document}

\title{Double Homotopy (Co)Limits for Relative Categories}

\author{Kay Werndli}
\address{Mathematisch Instituut, Universiteit Utrecht, Utrecht, Netherlands}
\email{k.r.werndli@uu.nl}
\thanks{The author was partially supported by Swiss NSF grant 200020 149118.}


\date{\today}

\begin{abstract}
  We answer the question to what extent homotopy (co)limits in categories with weak equivalences allow for a Fubini-type interchange law. The main obstacle is that we do not assume our categories with weak equivalences to come equipped with a calculus for homotopy (co)limits, such as a derivator.
\end{abstract}

\maketitle

\setcounter{section}{-1}
\section{Introduction}
In modern, categorical homotopy theory, there are a few different frameworks that formalise what exactly a ``homotopy theory'' should be. Maybe the two most widely used ones nowadays are Quillen's model categories (see e.g.~\cite{DwyerSpalinski} or~\cite{Hovey1999}) and $(\infty,1)$\hyph categories. The latter again comes in different guises, the most popular one being quasicategories (see e.g.~\cite{JoyalQuasiNotes} or~\cite{Lurie2009}), originally due to Boardmann and Vogt \cite{Boardman1973}. Both of these contexts provide enough structure to perform homotopy invariant analogues of the usual categorical limit and colimit constructions and more generally Kan extensions.

An even more stripped-down notion of a ``homotopy theory'' (that arguably lies at the heart of model categories) is given by just working with categories that come equipped with a class of weak equivalences. These are sometimes called {\it relative categories\/} \cite{Barwick2012}, though depending on the context, this might imply some restrictions on the class of weak equivalences. Even in this context, we can still say what a homotopy (co)limit is and do have tools to construct them, such as model approximations due to Chach\'olski and Scherer~\cite{HToD} or more generally, left and right deformation retracts as introduced in~\cite{DHKS} and generalised in section~\ref{sec:construction} below.

To actually obtain a calculus of homotopy (co)limits, maybe the most general structure is that of a {\it derivator\/} due to Franke, Heller \cite{Heller1988} and Grothendieck (see also \cite{Groth2013}). Every model category has an underlying derivator \cite{HToD,Cisinski2003} and therefore a good calculus of homotopy (co)limits. The same is true for a bicomplete $(\infty,1)$\hyph category, as shown by Riehl and Verity in \cite{Riehl_Verity-Kan_preprint}.

Now, for a relative category, we might be able to construct homotopy (co)limits but not a derivator structure. The goal of this paper is to answer the question to what extent such relative categories still allow for a Fubini-type homotopy (co)limit interchange law, even in the absence of a derivator structure.

For us, the main motivation to pursue such a result was that a proof based on universal properties, rather than explicit constructions, is much more conceptually satisfying. Even in nice contexts (e.g.~when working with model categories), explicit constructions of homotopy colimits can become quite involved when no additional hypotheses (such as cofibrant generation) are imposed (c.f.~\cite{HToD}). For example, checking a seemingly simple claim such as~\ref{prop:quillen functors preserve hocolims} with explicit constructions seems rather tedious. For completeness, we also provide a few examples \ref{ex:example1}, \ref{ex:example2},~\ref{ex:example3} immediately after our main theorem that are not nice model categories but where our theorem is applicable.

In section 1, we first derive a universal property of the (strict) localisation of a category that we are going to need and which is usually only proven for explicit constructions of the localisation (e.g.~for the homotopy category of a model category). In section 2, we restate the usual definition of a derived functor and, more importantly, of an absolute derived functor. We then give an alternative, external characterisation of these, which don't seem to have been previously known but turn out to be very useful to us later on.

In section 3, we slightly generalise the approach of \cite{DHKS} for the construction of derived functors and see that the usual construction of derived functors via (co)fibrant replacements is not confined to model categories. In sections 4 and 5, for lack of references, we quickly recall the definition, coherence and Beck-Chevalley interchange condition for the calculus of mates. These are standard tools in the world of derivators but less well-known outside of it.

In the section that follows and motivated by \cite{Maltsiniotis2007}, we investigate the concept of derived adjunctions in our general context and show the relation between absolute derived functors and adjoints thereof. It is here that we obtain our first homotopically substantial result.
\setcounter{section}{6}
\setcounter{theorem}{3}
\begin{corollary}
  Let~$F\dashv G$ be an adjunction of functors between categories with weak equivalences such that an absolute total right derived functor~$\Rbb G$ of~$G$ exists. Then~$F \dashv G$ is derivable if and only if\/~$\Rbb G$ has a left adjoint.\eop
\end{corollary}

Here, an adjunction between two categories with weak equivalences is called {\it derivable\/} if the left and right adjoint respectively have absolute total left and right derived functors and these form an adjunction between the corresponding homotopy categories (this generalises the {\it deformable adjunctions\/} of \cite{DHKS}). The example to keep in mind here are Quillen adjunctions between model categories.

As we learned later, the same result was already obtained in the yet unpublished \cite{Gonzalez2012} but using different techniques. In {\it op.~cit.}, the author used internal, diagrammatic methods, whereas we use our external characterisation of absolute Kan extensions obtained in section 2.

In section 7, we apply this result to the notion of homotopy (co)limits and obtain two equivalent definitions. Finally, sections 8 and 9 contain our main results about the pointwiseness of certain homotopy colimits. What we mean by this is best exemplified by the often used case of a double homotopy pushout. Given a $3\times 3$\hyph diagram of the form
\[ \xymatrix@R=1.4em@C=1.4em{ \bullet & \bullet \ar[l] \ar[r] & \bullet \\ \bullet \ar[u] \ar[d] & \bullet \ar[l]\ar[r]\ar[u]\ar[d] & \bullet \ar[u]\ar[d] \\ \bullet & \bullet \ar[l]\ar[r] & \bullet & *!<1.4em,.5ex>{,} } \]
most homotopy theorists take it for granted that the homotopy colimit of this diagram can be calculated as a double homotopy pushout in two ways: rows first or columns first. Under suitably nice hypotheses (e.g.~working in a cofibrantly generated model category) this is certainly true, though there doesn't seem to be a formal proof of this relying on universal properties. Rather, explicit formulae for the homotopy colimits or descriptions by cofibrant replacements are used.

It is important to note that the pointwiseness of homotopy colimits (say for double homotopy pushouts as above) consists of two parts; namely that the objects in the span obtained by the row- or columnwise homotopy colimit are what one would expect and that the same is true for the morphisms between them. Our first result concerning the pointwiseness of homotopy colimits is now the following.
\setcounter{section}{9}
\setcounter{theorem}{4}
\begin{theorem}\label{thm:main thm1}
  Let~$\Ceu$ be a category with weak equivalences, $\Ieu$,~$\Jeu$ two index categories and~$J \in \Ob\Jeu$ such that all derived adjunctions in the left-hand square
  \[ \vcenter{\xymatrix@C=3em{ \Ho\bigl((\Ceu^\Ieu)^\Jeu\bigr) \cong \Ho\bigl((\Ceu^\Jeu)^\Ieu\bigr) \ar@<-2.93em>[d]_{\Ho\ev_J}_{}="21" \ar@<1ex>[r]^-{\hocolim_\Ieu}^-{}="1" & \Ho(\Ceu^\Jeu) \ar@<-1ex>[d]_{\Ho\ev_J}_{}="31" \ar@<1ex>[l]^-{\Ho\Delta}^-{}="2" \\ 
*+!<2.5em,0mm>{\Ho(\Ceu^\Ieu)} \ar@<2.07em>[u]_{\Rbb J_*}^{}="22" \ar@<1ex>[r]^-{\hocolim_\Ieu}^-{}="11" & \Ho\Ceu \ar@<1ex>[l]^-{\Ho\Delta}^-{}="12" \ar@<-1ex>[u]_{\Rbb J_*}^{}="32"
  \ar@{}"1";"2"|(.3){\hbox{}}="3" \ar@{}"1";"2"|(.7){\hbox{}}="4" \ar@{|-} "4";"3"
  \ar@{}"11";"12"|(.3){\hbox{}}="13" \ar@{}"11";"12"|(.7){\hbox{}}="14" \ar@{|-} "14";"13"
  \ar@{}"21";"22"|(.3){\hbox{}}="23" \ar@{}"21";"22"|(.7){\hbox{}}="24" \ar@{|-} "24";"23"
  \ar@{}"31";"32"|(.3){\hbox{}}="33" \ar@{}"31";"32"|(.7){\hbox{}}="34" \ar@{|-} "34";"33"
  }} \ \vcenter{\xymatrix{ \Ho\bigl((\Ceu^\Ieu)^\Jeu\bigr) \cong \Ho\bigl((\Ceu^\Jeu)^\Ieu\bigr) \ar@<-2.93em>[d]^{\Ho\ev_J} & \Ho(\Ceu^\Jeu) \ar[d]_{\Ho\ev_J}\ar[l]_-{\Ho\Delta} \\ *+!<2.5em,0mm>{\Ho(\Ceu^\Ieu)} & \Ho\Ceu \ar[l]^-{\Ho\Delta}}} \]
  exist and\/~$\Ho\Delta$ composes with\/~$\Rbb J_*$. Then the right-hand square (filled with the identity) satisfies the (horizontal) Beck-Chevalley condition. In particular, there is an isomorphism
  \[ \sigma\colon {\hocolim_\Ieu}\circ\Ho\ev_J \cong \Ho\ev_J\circ{\hocolim_\Ieu}, \qquad\text{natural in }J \in \Ob\Jeu. \]
\end{theorem}

Here, the stated hypothesis that $\Ho\Delta$ {\it compose\/} with~$\Rbb J_*$ means that the composite of right-derived functors $\Rbb J_*\circ\Ho\Delta$ is the right-derived functor of the composite $J_*\circ\Delta$; something which is automatic when only considering Quillen pairs but not in our general context.

It is this theorem~\ref{thm:main thm1}, which allows us to conclude that the most common homotopy colimits, such as double homotopy pushouts, coproducts and telescopes can always be calculated pointwise, assuming they exist and, curiously, assuming the base category has a terminal object.

In the above theorem, applied to the double homotopy pushout example, the existence of an isomorphism $\hocolim_\Ieu\circ\Ho\ev_J \cong \Ho\ev_J\circ\hocolim_\Ieu$ ensures that the objects of the row- or columnwise homotopy colimit are the ones one would expect, while the naturality in~$J$ ensures the same for the morphisms between these objects.

It is this essential naturality condition at the end, which is hard to prove. Since the mixing of homotopy colimits and right derived functors can be inconvenient (as one might be able to construct left derived functors such as homotopy colimits but not right derived ones), we then go on to investigate how we can use the left adjoints to evaluation functors to obtain the sought for naturality condition. To that extent, we have the following result.
\setcounter{section}{9}
\setcounter{theorem}{15}
\begin{theorem}
  Let~$\Ceu$ be a category with weak equivalences and $\Ieu$,~$\Jeu$ two index categories such that for all~$J \in \Ob\Jeu$, the two evaluation functors~$\ev_J$ indicated below have fully faithful left adjoints~$J_!$ and the derived adjunctions in the square
  \[ \xymatrix@C=3.5em{ \Ho\bigl((\Ceu^\Ieu)^\Jeu\bigr)  \cong \Ho\bigl((\Ceu^\Jeu)^\Ieu\bigr) \ar@<-2.07em>[d]^{\Ho\ev_J}^{}="22" \ar@<1ex>[r]^-{\hocolim_\Ieu}^-{}="1" & \Ho(\Ceu^\Jeu) \ar@<1ex>[d]^{\Ho\ev_J}^{}="32" \ar@<1ex>[l]^-{\Ho\Delta}^-{}="2" \\ 
  *+!<2.5em,0mm>{\Ho(\Ceu^\Ieu)} \ar@{^{c}->}@<2.93em>[u]^{\Lbb J_!}^{}="21" \ar@<1ex>[r]^-{\hocolim_\Ieu}^-{}="11" & \Ho\Ceu \ar@<1ex>[l]^-{\Ho\Delta}^-{}="12" \ar@{^{c}->}@<1ex>[u]^{\Lbb J_!}^{}="31"
  \ar@{}"1";"2"|(.3){\hbox{}}="3" \ar@{}"1";"2"|(.7){\hbox{}}="4" \ar@{|-} "4";"3"
  \ar@{}"11";"12"|(.3){\hbox{}}="13" \ar@{}"11";"12"|(.7){\hbox{}}="14" \ar@{|-} "14";"13"
  \ar@{}"21";"22"|(.3){\hbox{}}="23" \ar@{}"21";"22"|(.7){\hbox{}}="24" \ar@{|-} "24";"23"
  \ar@{}"31";"32"|(.3){\hbox{}}="33" \ar@{}"31";"32"|(.7){\hbox{}}="34" \ar@{|-} "34";"33"
  } \]
  exist and there is an isomorphism~$\hocolim_\Ieu\circ\Ho\ev_J \cong \Ho\ev_J\circ\hocolim_\Ieu$. Then there is a family of isomorphisms
  \[ \beta_J\colon \hocolim_\Ieu\circ\Ho\ev_J \cong \Ho\ev_J\circ\hocolim_\Ieu \qquad\text{natural in~$J$.} \]
\end{theorem}

\setcounter{section}{0}
\section{Localisations Reviewed}
By the localisation of a category~$\Ceu$ with respect to a class of morphisms~$\Wcal$ is meant the (comparison functor with) the category (possibly not locally small) obtained by formally inverting the morphisms in~$\Wcal$.
\begin{definition}
  Let~$\Ceu$ be a category and~$\Wcal$ a class of arrows in~$\Ceu$ (which we usually call the {\it weak equivalences\/}\index{Equivalences!weak}\index{Weak!equivalences} of~$\Ceu$). A {\it weak localisation\/}\index{Localisation!of a category}\index{Localisation!weak}\index{Weak!localisation} of~$\Ceu$ is a functor $H\colon \Ceu \to \Ceu[\Wcal^{-1}]$ sending all arrows in~$\Wcal$ to isomorphisms and having the following universal property: 
  \begin{enumerate}
    \item Whenever we have a functor~$F\colon \Ceu \to \Deu$ sending all arrows in~$\Wcal$ to isomorphisms, then there is some~$\ol F\colon \Ceu[\Wcal^{-1}] \to \Deu$ together with a natural isomorphism~$F \cong \ol F\circ H$;
    \item For every category~$\Deu$, the precomposition~$H^*\colon \Deu^{\Ceu[\Wcal^{-1}]} \to \Deu^\Ceu$ is fully faithful. That is to say, for any two $F$,~$G\colon \Ceu[\Wcal^{-1}] \to \Deu$ and any transformation~$\tau\colon F\circ H \rar G\circ H$, there is a unique~$\ol\tau\colon F \rar G$ such that~$\tau = \ol\tau_H$.
  \end{enumerate}
\end{definition}
\begin{observation}
  One easily sees that~$\Ceu[\Wcal^{-1}]$ is unique up to equivalence; that the extension~$\ol F$ in property (1) is unique up to unique isomorphism and the natural isomorphism $F \cong \ol F\circ H$ is unique up to unique automorphism of~$\ol F$.
\end{observation}

Most authors will define localisations differently. Namely, $\ol F$ in (1) needs to be unique and the natural isomorphism $F \cong \ol F\circ H$ is required to be an identity, while property (2) is left out entirely. We call this a {\it strict localisation\/}.
\begin{definition}
  A {\it (strict) localisation\/}\index{Localisation!strict}\index{Strict!localisation} of a category~$\Ceu$ at a class of arrows~$\Wcal$ (again called the {\it weak equivalences\/}) is a functor~$H\colon \Ceu \to \Ceu[\Wcal^{-1}]$ that sends arrows in~$\Wcal$ to isomorphisms and such that every functor~$F\colon \Ceu \to \Deu$ that does so factors uniquely through~$H$ as $F = \ol F \circ H$. It follows that~$\Ceu[\Wcal^{-1}]$ is unique up to isomorphism. We will seldomly have the situation where we are given two different classes of weak equivalences in~$\Ceu$ and will thus just write~$\Ho\Ceu \defas \Ceu[\Wcal^{-1}]$, leaving the class~$\Wcal$ implicit.

\end{definition}
\begin{example}
  If every arrow in~$\Wcal$ is already an isomorphism, then~$\id_\Ceu\colon \Ceu \to \Ceu$ is a localisation of~$\Ceu$ at~$\Wcal$. In particular, if~$\Ceu$ is a groupoid, then~$\id_\Ceu$ is a localisation at any class of weak equivalences.
\end{example}
\begin{example}
  If~$\Ceu$ is a model category, then the canonical functor~$\Ceu \to \Ho\Ceu$ is a localisation of~$\Ceu$ with respect to its weak equivalences. In fact, having well-behaved localisations is the reason why the theory of model categories was developed in the first place.
\end{example}
\begin{convention}
  If~$\Wcal$ is a class of weak equivalences in a category~$\Ceu$, one can always add the isomorphisms of~$\Ceu$ to~$\Wcal$ without changing the localisation and so we shall always assume that isomorphisms are weak equivalences.
\end{convention}

It seems curious that part (2) from the definition of a localisation is omitted entirely in the strict version. After all, it has some consequences, e.g.~for derived functors, which seem to be important enough to be proven for an explicit construction of a localisation (e.g.~\cite[5.9]{DwyerSpalinski} for the homotopy category of a model category). But using an explicit construction is not necessary as the following proof shows.%
\begin{proposition}\label{prop:strict localisation is weak}
  If~$H\colon \Ceu \to \Ho\Ceu$ is a strict localisation of a category~$\Ceu$ at a class~$\Wcal$ of morphisms, then, for every category~$\Deu$, the precomposition functor~$H^*\colon \Deu^{\Ho\Ceu} \to \Deu^\Ceu$ is fully faithful.
\end{proposition}
\begin{proof}
  Let~${[1]}$ be the {\it interval category\/}\index{Interval!category}\index{Category!interval} with two objects $0$,~$1$ and exactly one non-identity morphism $i\colon 0 \to 1$. There is the canonical twist isomorphism $(\Deu^\Ceu)^{[1]} \cong (\Deu^{[1]})^\Ceu$ so that pairs of functors $F$,~$G\colon \Ceu \to \Deu$ together with~$\tau\colon F \rar G$ correspond to functors~$\Ceu \to \Deu^{[1]}$. Now given $F$,~$G\colon \Ho\Ceu \to \Deu$ together with $\tau\colon F\circ H \rar G\circ H$, these determine
  \[ T\colon \Ceu \to \Deu^{[1]}\text{ with }(TC)i = \tau_C \text{ and }(Tf)_0 = FHf,\, (Tf)_1 = GHf\text{ for }f\colon C \to C'. \]
  Under~$T$, arrows~$f \in \Wcal$ are sent to isomorphisms because~$HFf$ and~$GFf$ are invertible and so we get a unique~$\ol T\colon \Ho\Ceu \to \Deu^{[1]}$ such that~$T = \ol T\circ H$. Under the above twist isomorphism, this corresponds to a unique pair of functors
  \[ F',\, G'\colon \Ho\Ceu \to \Deu \quad\text{defined by}\quad F' = \ev_0\circ\ol T \text{ and } G' = \ev_1\circ\ol T \]
  together with $\ol\tau\colon F' \rar G'$ defined by $\ol\tau_C = (\ol TC)i$ that satisfies~$\tau = \ol\tau_H$.
\end{proof}

\begin{nomenclature}
  For brevity reasons, if~$\Ceu$ and~$\Deu$ are categories with weak equivalences, we call a functor~$F\colon \Ceu \to \Deu$ {\it homotopical\/}\index{Homotopical!functor}\index{Functor!homotopical} iff it preserves weak equivalences. Moreover, by a {\it natural weak equivalence\/}\index{Weak equivalence!natural}\index{Natural!weak equivalences}, we mean a natural transformation~$\tau\colon F \rar G$, all of whose components are weak equivalences.
\end{nomenclature}

Assuming we have an explicit localisation construction (e.g.~using zig-zags \cite{Borceux1}), we note that~$\Ho$ is strictly 2\hyph functorial. In particular, any adjunction $F\dashv G$ of homotopical functors yields an adjunction $\Ho F\dashv\Ho G$ between the correspondig homotopy categories.
\begin{observation}
  Let~$\Ceu$ be a category with weak equivalences and~$H\!\colon\! \Ceu \to \Ho\Ceu$ a localisation. If an object~$0$ is initial in~$\Ceu$, then~$H0$ is initial in~$\Ho\Ceu$.
\end{observation}
\begin{proof}
  Both adjoints in~$0\colon \{\ast\} \rightleftarrows \Ceu\lon\bang$ are homotopical.
\end{proof}
\section{Derived Functors}
  Consider a category $\Ceu$ with weak equivalences. Time and again, one finds oneself in the situation where one wants to study some~$F\colon \Ceu \to \Deu$ that doesn't map weak equivalences to isomorphisms (so that it cannot be extended along the localisation). Sometimes this is conceived as a defect of~$F$ (e.g.~for~$F$ a [co]limit functor) and at other times as an interesting peculiarity that can be exploited to construct invariants (e.g.~using $\Hom$\hyph functors, tensoring or global sections). In either case, the approach is usually to approximate~$F$ as well as possible by a functor on~$\Ho\Ceu$.%
\begin{definition}
  Let~$H\colon \Ceu \to \Ho\Ceu$ be a localisation of a category with weak equivalences and $F\colon \Ceu \to \Deu$. Recall that a {\it left derived functor\/}\index{Functor!left derived}\index{Derived!functor}\index{Left!derived functor}~$(LF,\lambda\colon LF\circ H \rar F)$ of~$F$ is a right Kan extension of~$F$ along~$H$. That is, for every~$L\colon \Ho\Ceu \to \Deu$ together with~$\tau\colon L\circ H \rar F$, there is a unique transformation~$\ol\tau\colon L \rar LF$ such that~$\lambda\circ\ol\tau_H = \tau$. The universal 2\hyph arrow~$\lambda$ is called the {\it counit\/}\index{Counit!of a left derived functor} of the left derived functor. Dually (more precisely ``{\it co\/}-dually'', i.e.~just reversing 2-cells), one defines a {\it right derived functor\/}\index{Functor!right derived}\index{Right!derived functor}~$(RF,\rho)$ of~$F$, whose universal 2\hyph arrow is called its {\it unit\/}\index{Unit!of a right derived functor}.

  If~$\Deu$ also comes with a class of weak equivalences and a localisation functor~$H'\colon \Deu \to \Ho\Deu$, we will usually be more interested in the {\it total left derived functor\/}\index{Deriverd functor!total left}\index{Left!derived functor!total}, which is the left derived functor~$\Lbb F \defas L(H'\circ F)$ of~$H'\circ F$. Dually for the {\it total right derived functor\/}\index{Derived functor!total right}\index{Right!derived functor!total}~$\Rbb F \defas R(H'\circ F)$.%
\end{definition}
\begin{remark}
  Being a terminal object in the category~$H^*\comma F$, a left derived functor~$(LF,\lambda)$ is unique up to unique isomorphism. In particular, for a fixed~$LF$, its counit~$\lambda$ is unique up to precomposition with~$\alpha_H$ for some unique automorphism~$\alpha$ of~$LF$.
\end{remark}
  As a special case of this universal property, we obtain that the operations~$L$ (resp.\ $\Lbb$) and~$R$ (resp.\ ~$\Rbb$) are functorial.
\begin{definition}
  Let $F$,~$F'\colon \Ceu \to \Deu$ have right Kan extensions~$(LF,\lambda)$ and $(LF',\lambda')$ along some~$H\colon \Ceu \to \Heu$ and let~$\sigma\colon F \rar F'$ be a natural transformation. Then there is a unique~$L\sigma \defas \ol{\sigma\circ\lambda}\colon LF \rar LF'$ such that~$\sigma\circ\lambda = \lambda'\circ (L\sigma)_H$. Dually, if $G$,~$G'\colon \Ceu \to \Deu$ have left Kan extensions~$(RG,\rho)$ and~$(RG',\rho')$ along~$H$ and~$\tau\colon G \rar G'$, then there is a unique~$R\tau \defas \ol{\rho'\circ\sigma}\colon RG \rar RG'$ such that $(R\tau)_H\circ\rho = \rho'\circ\tau$. The unicity of these shows that~$L$ and~$R$ strictly preserve vertical compositions of natural transformations as well as identities. Similarly for the totally derived versions~$\Lbb\sigma \defas L(H'\sigma)$ and~$\Rbb\tau \defas R(H'\tau)$ for $H'\colon \Deu\to \Heu'$.
\end{definition}

There is the following convenient external characterisation of derived functors. Unfortunately, there are some class-theoretic difficulties involved (cf.\ the remark below).
\begin{proposition}\label{prop:external characterisation of derived functors}
  If~$(LF,\lambda)$ is a right Kan extension of~$F\colon \Ceu \to \Deu$ along some functor~$H\colon \Ceu \to \Heu$, then
  \[ \varphi_{X}\colon \Nat(X,LF) \to \Nat(X\circ H, F),\, \tau \mapsto \lambda\circ\tau_H \]
  is a bijection, natural in~$X\colon \Heu \to \Deu$. Conversely, if there is a functor~$LF\colon \Heu \to \Deu$ together with a natural family of bijections
  \[ \varphi_{X}\colon \Nat(X,LF) \to \Nat(X\circ H,F) \]
  indexed by all functors~$X\colon \Heu \to \Deu$, then~$(LF,\varphi_{LF}\id_{LF})$ is a left derived functor of~$F$. Moreover, these two constructions are inverse to each other.
\end{proposition}
\begin{proof}
  The map~$\tau \mapsto \lambda\circ(H\tau)$ being a bijection is just the universal property of a right Kan extension and naturality is easy. Conversely, if there is a natural family~$(\varphi_{X})_{X}$ of bijections as in the proposition, we put~$\lambda \defas \varphi_{LF}\id_{LF}$. Now for any~$X\colon \Heu \to \Deu$ together with~$\tau\colon X\circ H \rar F$ there is a unique~$\ol\tau\colon X \rar LF$ such that~$\varphi_{X}\ol\tau = \tau$. Chasing~$\id_{LF}$ around the naturality square
  \begin{samepage}\begin{equation}\label{eqn:naturality of kan bijections} \vcenter{\xymatrix{ \Nat(LF,LF) \ar[d]_{{\ol\tau}^*} \ar[r]^-{\varphi_{LF}} & \Nat(LF\circ H, F) \ar[d]^{{\ol\tau}_H^*} \\ \Nat(X,LF) \ar[r]_-{\varphi_{X}} & \Nat(X\circ H, F) & *!<2.2em,.5ex>{,} }} \end{equation}
    we see that~$\ol\tau$ is the unique transformation~$X \rar LF$ such that~$\lambda\circ\ol\tau_H = \tau$.\end{samepage}
\end{proof}
\begin{remark}
  In ordinary NBG class theory, one cannot properly formalise the above proposition because in general, there is no class of all functors~$X\colon \Heu \to \Deu$. The obvious remedies for this (apart from never using the external characterisation) is to assume the existence of a universe or switch to a higher class theory that has 2\hyph classes, whose relation to ordinary classes is the same as that of classes to sets (such as in \cite{JoyCats}).
\end{remark}

Of particular importance for homotopy theory are the so-called {\it absolute\/} Kan extensions, i.e.~those preserved by any morphism, as was observed by Maltsiniotis~\cite{Maltsiniotis2007}. Their importance comes from the fact that all derived functors that arise from Quillen adjunctions are absolute and that there is a nice interplay between adjunctions and absolute derived functors as we shall see in section~\ref{sec:derived adjunctions}.
\begin{definition}
  A right Kan extension~$(LF,\lambda)$ of a functor~$F\colon \Ceu \to \Deu$ along $H\colon \Ceu \to \Heu$ (and in particular a [total] left derived functor) is called {\it absolute\/}\index{Kan extension!absolute}\index{Derived functor!absolute}\index{Left!derived functor!absolute}\index{Right!derived functor!absolute}\index{Absolute!derived functor}\index{Left!Kan extension!absolute}\index{Right!Kan extension!absolute}\index{Absolute!Kan extension} iff for every~$Y\colon \Deu \to \Eeu$, the composite~$(Y\circ LF,Y\lambda)$ is a right Kan extension of~$Y\circ F$ along~$H$. Dually for left Kan extensions (and in particular [total] right derived functors).
\end{definition}
\begin{example}
  Given two functors~$F\colon \Ceu \rightleftarrows \Deu\lon G$, then~$F \dashv G$ with unit~$\eta$ iff~$(G,\eta)$ is an absolute left Kan extension of~$\id_\Ceu$ along~$F$ (cf.\ \ref{prop:adjunctions and absolute Kan extensions}). Dually, $F \dashv G$ with counit~$\varepsilon$ iff~$(F,\varepsilon)$ is an absolute right Kan extension of~$\id_\Deu$ along~$G$.
\end{example}
\begin{example}
  If~$\Ceu$ and~$\Deu$ are two categories, each equipped with a class of weak equivalences and~$F\colon \Ceu \to \Deu$ is homotopical, then~$\Ho F\colon \Ho\Ceu \to \Ho\Deu$ together with the identity transformation is both an absolute total left and an absolute total right derived functor of~$F$.
\end{example}
\begin{remark}
  We have already seen that taking Kan extensions preserves vertical compositions of natural transformations. In the absolute case, this is also true for horizontal compositions (which doesn't even make sense in the non-absolute case). To wit, consider two pairs of parallel functors $F$,~$F'\colon \Ceu \to \Deu$, $Y$,~$Y'\colon \Deu \to \Eeu$ together with $\sigma\colon F \rar F'$,~$\tau\colon Y \rar Y'$ such that~$F$ and~$F'$ have absolute right Kan extensions~$(LF,\lambda)$ and~$(LF',\lambda')$ along an~$H\colon \Ceu \to \Heu$. Then $L(\tau\star\sigma) = \tau\star L\sigma$ (where $\star$ is horizontal composition of natural transformations).
\end{remark}
\begin{proof} For an arbitrary object~$C \in \Ob{\Ceu}$, we easily calculate
  \begin{align*}
    Y'\lambda'_C \circ (\tau\star L\sigma)_{HC} &\stackrel{\hphantom{\text{\tiny $\tau$ nat}}}{=} Y'\lambda'_C\circ Y'(L\sigma)_{HC}\circ\tau_{(LF)HC} = Y'\sigma_C\circ Y'\lambda_C\circ\tau_{(LF)HC} \\
      &\stackrel{\text{\tiny $\tau$ nat}}{=} Y'\sigma_C\circ\tau_{FC}\circ Y\lambda_C = (\tau\star\sigma)_C\circ G\lambda_C.
    \end{align*}
\end{proof}

  As an obvious next step, we can adapt the external characterisation \ref{prop:external characterisation of derived functors} to the case of absolute derived functors (again stated in larger generality). 
\begin{proposition}  
  If~$(LF,\lambda)$ is an absolute right Kan extension of~$F\colon \Ceu \to \Deu$ along some functor~$H\colon \Ceu \to \Heu$, then
  \[ \varphi_{\Eeu,X,Y}\colon \Nat(X,Y\circ LF) \to \Nat(X\circ H,Y\circ F),\, \tau \mapsto Y\lambda\circ\tau_H \]
  is a bijection natural in~$\Eeu$,~$X\colon \Heu \to \Eeu$ and~$Y\colon \Deu \to \Eeu$. Conversely, if there is~$LF\colon \Heu \to \Deu$ together with a family of bijections
  \[ \varphi_{\Eeu,X,Y}\colon \Nat(X,Y\circ LF) \to \Nat(X\circ H,Y\circ F) \]
  natural in the category~$\Eeu$ and the functor~$X\colon \Heu \to \Eeu$ (naturality in~$Y\colon \Deu \to \Eeu$ is automatic), then~$(LF,\varphi_{\Deu,LF,\id_\Deu}\id_{LF})$ is a total left derived functor of~$F$. Moreover, these two constructions are inverse to each other.
\end{proposition}
\begin{proof}
  The bijectivity of~$\tau \mapsto Y\lambda\circ\tau_H$ is just the universal property of the right Kan extension~$(Y\circ LF,Y\lambda)$ and naturality is easy. For the converse claim, we put~$\lambda \defas \varphi_{\Deu,LF,\id_\Deu}\id_{LF}$. For $Y\colon \Deu \to \Eeu$,~$X\colon \Heu \to \Eeu$ and~$\tau\colon X\circ H \rar Y\circ F$ there is a unique~$\ol\tau\colon X \rar Y\circ LF$ such that~$\varphi_{\Eeu,X,Y}\ol\tau = \tau$ and chasing~$\id_{LF}$ around the commutative diagram
  \[ \xymatrix@C=4em{ \Nat(LF,LF) \ar[r]^-{\varphi_{\Deu,LF,\id_\Deu}} \ar[d]_{Y_*} & \Nat(LF\circ H, F) \ar[d]^{Y_*} \\ \Nat(Y\circ LF, Y\circ LF) \ar[r]^-{\varphi_{\Eeu,Y\circ LF,Y}} \ar[d]_{\ol\tau^*} & \Nat(Y\circ LF\circ H, Y\circ F) \ar[d]^{\ol\tau_H^*} \\ \Nat(X,Y\circ LF) \ar[r]_-{\varphi_{\Eeu,X,Y}} & \Nat(X\circ H, Y\circ F) } \]
  we obtain that indeed,~$\ol\tau$ is the unique transformation satisfying~$\tau = Y\lambda\circ\ol\tau_H$.
\end{proof}

  Because we shall usually work with total derived functors, let us quickly restate this theorem for the total derived case.
\begin{corollary}
  Let $\Ceu$,~$\Deu$ be categories with weak equivalences and $H_\Ceu$,~$H_\Deu$ the corresponding localisations. Moreover, let~$F\colon \Ceu \to \Deu$ and\/ $\Lbb F\colon \Ho\Ceu \to \Ho\Deu$. If\/~$\Lbb F$ is an absolute total left derived functor of\/~$F$ with counit~$\lambda$, then~$\tau \mapsto Y\lambda\circ\tau_{H_\Ceu}$ defines a bijection
  \[ \varphi_{\Eeu,X,Y}\colon \Nat(X,Y\circ\Lbb F) \cong \Nat(X\circ H_\Ceu, Y\circ H_\Deu\circ F) \]
  natural in $\Eeu$,~$X\colon \Ho\Ceu \to \Eeu$ and~$Y\colon \Ho\Deu \to \Eeu$. Conversely, if there is such a natural family of bijections, then~$\Lbb F$ is an absolute total left derived functor of\/~$F$ with counit $\varphi_{\Ho\Deu,\Lbb F,\id_{\Ho\Deu}} \id_{\Lbb F}$. These assignments are mutually inverse. Dually, if we have~$G\colon \Deu \to \Ceu$ and~$\Rbb G\colon \Ho\Deu \to \Ho\Ceu$ such that\/~$\Rbb G$ is an absolute total right derived functor of\/~$G$ with unit~$\rho$, then~$\tau \mapsto \tau_{H_\Deu}\circ E'\rho$ defines a bijection
  \[ \psi_{\Eeu',E',Y}\colon \Nat(E'\circ\Rbb G, Y) \cong \Nat(E'\circ H_\Ceu\circ G, Y\circ H_\Deu) \]
  natural in $\Eeu'$, $E'\colon \Ho\Ceu \to \Eeu'$,~$Y\colon \Ho\Deu \to\Eeu'$ and conversely given such a natural family of bijections, $\Rbb G$ is an absolute total right derived functor of\/~$G$ with unit~$\psi_{\Ho\Ceu,\id_{\Ho\Ceu},\Rbb G}\id_{\Rbb G}$. Again, these assignments are mutually inverse. \eop
\end{corollary}

\section{Construction of Derived Functors}\label{sec:construction}
Virtually all derived functors that occur ``in nature'' are constructed by means of resolutions (a.k.a.~replacements, a.k.a.~approximations) and we shall quickly abstract these constructions to our context. Our treatment has been greatly inspired by \cite{DHKS} and our explicit goal is to remove a functoriality condition, which is present in {\it op.~cit}. For this, we fix some category~$\Ceu$ equipped with a class of weak equivalences and write~$H\colon \Ceu \to \Ho\Ceu$ for its localisation.

\begin{convention}
  If not stated otherwise, we will always equip a subcategory~$\Ceu_0 \subseteq \Ceu$ with the weak equivalences coming from~$\Ceu$. I.e.\ an arrow in~$\Ceu_0$ is a weak equivalence iff it is one in~$\Ceu$.
\end{convention}

\begin{definition}
  Following the nomenclature in \cite{DHKS} (and weakening their notion) a {\it left deformation retract\/}\index{Left!deformation retract}\index{Deformation!retract}\index{Retract!left deformation} of~$\Ceu$ is a full subcategory~$\Ceu_0 \hookrightarrow \Ceu$ with localisation~$H_0\colon \Ceu_0 \to \Ho\Ceu_0$ such that there exist
  \begin{enumerate}
    \item a map~$Q\colon \mathop{\rm Ob}\Ceu \to \mathop{\rm Ob}{\Ceu_0}$ and for every~$f\colon C \to C'$ in~$\Ceu$ a~$Qf\colon QC \to QC'$ in~$\Ceu_0$ such that~$H_0Q$ defines a functor~$H_0Q\colon \Ceu \to \Ho\Ceu_0$ that sends weak equivalences to isomorphisms and thus induces a unique~$\widetilde Q\colon \Ho\Ceu \to \Ho\Ceu_0$ such that~$\widetilde Q\circ H = H_0Q$;
    \item for each $C \in \Ob\Ceu$ a weak equivalence $q_C\colon QC \to C$ such that for every~$f\colon C \to C'$
      \[ \vcenter{\xymatrix{ QC \ar[d]_{q_C} \ar[r]^{Qf} & QC' \ar[d]^{q_{C'}} \\ C \ar[r]_{f} & C' }} \qquad\text{commutes in~$\Ceu$.} \]
  \end{enumerate}
  The triple~$(\Ceu_0,Q,q)$ is then called a {\it left deformation retraction\/}\index{Left!deformation retraction}\index{Deformation!retraction}\index{Retraction!left deformation} (of~$\Ceu$ to~$\Ceu_0$). Dually, one defines a {\it right deformation retract\/}\index{Right!deformation retract}\index{Retract!right deformation}\index{Right!deformation retraction}\index{Retraction!right deformation}.
\end{definition}

\begin{observation}
  If the weak equivalences in~$\Ceu$ satisfy 2\hyph out\hyph of\hyph 3, then the requirement that~$H_0Q$ send weak equivalences to isomorphisms is superfluous. Indeed, for~$f\colon C \to C'$ in~$\Ceu$ we have $q_C'\circ Qf = f \circ q_C$ and so~$f$ is a weak equivalence iff~$Qf$ is.
\end{observation}

  As already mentioned, our notion of a deformation retract is weaker than the one in \cite{DHKS}. There, a left deformation retract is defined as a full subcategory~$I\colon \Ceu_0 \hookrightarrow\Ceu$ together with a homotopical functor~$Q\colon \Ceu \to \Ceu_0$ and a natural weak equivalence~$q\colon I\circ Q \rar \id_\Ceu$. One reason why  one would like to have this stronger condition is the following.
\begin{remark}\label{rem:functorial deformations}
  If~$Q$ in the above definition is a functor $\Ceu \to \Ceu_0$, we can lift the deformation retract to diagram categories. To wit, if $\Ceu_0 \subseteq \Ceu$ is a left deformation retract as above, with $Q\colon \Ceu \to \Ceu_0$ a functor and~$\Ieu$ is a small index category, then $\Ceu_0^{\ \Ieu} \subseteq \Ceu^\Ieu$ is a left deformation retract by applying~$Q$ pointwise (though this is not enough to obtain homotopy colimits). This is not possible if~$Q$ only becomes a functor when passing to the homotopy category, since $\Ho(\Ceu^\Ieu) \not\cong \Ho(\Ceu)^\Ieu$ in general.
\end{remark}

\begin{example}\label{ex:cofibrant objects deformation retract}
  For~$\Meu$ a model category (not necessarily with functorial factorisations), the category of cofibrant objects~$\Meu_c$ together with some chosen cofibrant replacements~$q_C\colon QC \to C$ (i.e.~$QC$ is cofibrant and~$q_C$ an acyclic fibration) and chosen lifts~$Qf\colon QC \to QC'$ for~$f\colon C \to C'$ forms a left deformation retraction. Indeed, as shown in~\cite[Lemma 5.1]{DwyerSpalinski}, if $f\colon C \to C'$ is any morphism in~$\Meu$ and $q_C\colon QC \to C$, $q_{C'}\colon QC' \to C'$ are fixed cofibrant replacements, a dotted morphism making the square
  \[ \xymatrix{ QC \ar@{->>}[d]_{q_C}^{\sim} \ar@{..>}[r] & QC' \ar@{->>}[d]^{q_{C'}}_{\sim} \\ C \ar[r]_{f} & C' } \]
  commute in~$\Meu$ always exists (i.e.~we can choose a~$Qf$ as required) and is unique up to (left) homotopy. In particular, all such lifts are mapped to the same morphism under~$H_0\colon \Meu_c \to \Ho\Meu_c$ (cf.~\cite[Lemma 5.10]{DwyerSpalinski}, whose proof goes through even though~$\Meu_c$ need not be a model category) and thus~$H_0Q$ is a functor.
  
  If~$\Meu$ has functorial factorisations, then we even obtain the stronger version of a left deformation retraction, where $Q\colon \Meu \to \Meu_c$ is actually a functor.
\end{example}

\begin{example}\label{ex:deformation retract for left model approximations}
  More generally, if $L\colon \Meu \rightleftarrows \Ceu\lon R$ is a {\it left model approximation\/} in the sense of \cite{HToD} (left adjoint on the left and again with~$\Meu$ not necessarily having functorial factorisations), we let $\Ceu_0 \subseteq \Ceu$ be the full subcategory comprising all (objects isomorphic to) images of cofibrant objects in~$\Meu$ under~$L$. Choosing a cofibrant replacement~$Q$ for~$\Meu$ as in the last example
  \[ (C \xto{f} C') \quad\mapsto\quad (LQRC\xto{LQRf}LQRC') \quad\text{together with the}\quad LQRC \xto{q_{RC}^\flat} C \]
  defines a left deformation retraction of~$\Ceu$ to~$\Ceu_0$. Indeed, the $q_{RC}^\flat$ are weak equivalences because their adjuncts are and by definition of a left model approximation. Moreover, our cofibrant replacement is functorial when passing to the homotopy category because~$H_0LQR$ is just
  \[ \Ceu \xto{R} \Meu \xto{H_{\Meu}} \Ho\Meu \xto{\widetilde Q} \Ho\Meu_c \xto{\Ho L} \Ho\Ceu_0 \]
  (where for the last functor, we used that~$L$ sends weak equivalences between cofibrant objects to weak equivalences). Finally, the ``naturality squares'' for the $q_{RC}^\flat\colon LQRC\to C$ commute because their adjunct squares do.
\end{example}

  Although it is usually convenient to require that~$Q$ itself be a functor, this has the undesirable consequence of needing to have functorial factorisations on a model category for the theory to apply. For the construction of derived functors, it is much more important that~$Q$ become functorial when passing to the homotopy category.
\begin{proposition}
 Let~$(\Ceu_0,Q,q)$ be a left deformation retraction of~$\Ceu$ and $I\colon \Ceu_0 \hookrightarrow \Ceu$ the inclusion. Then the families~$(Hq_C)_{C \in \Ob\Ceu}$ and~$(H_0q_{C_0})_{C_0\in\Ob{\Ceu_0}}$ define natural isomorphisms%
  \[ \Ho I\circ\widetilde Q \cong \id_{\Ho\Ceu} \qquad\text{and}\qquad \widetilde Q\circ\Ho I \cong \id_{\Ho\Ceu_0}. \]
\end{proposition}
\begin{proof}
  By~\ref{prop:strict localisation is weak}, it suffices to check that the two families define natural isomorphisms
  \[ \Ho I\circ\widetilde Q\circ H = \Ho I\circ H_0Q \xRightarrow{\cong\;} H \qquad\text{and}\qquad \widetilde Q\circ\Ho I\circ H_0 = H_0Q\circ I \xRightarrow{\cong\;} H_0, \]
  which is simple because we already have ``naturality squares'' in~$\Ceu$. For example, the naturality of the first family corresponds to the commutativity of
  \[ \xymatrix@C=2em{ C \ar[d]_{f} & (\Ho I)\widetilde QHC \ar@{=}[r] \ar[d]_{(\Ho I)\widetilde QHf} & (\Ho I)H_0QC \ar@{=}[r] \ar[d]_{(\Ho I)H_0Qf} & HIQC \ar@{=}[r] \ar[d]_{HIQf} & HQC \ar[r]^-{Hq_C} \ar[d]_{HQf} & HC \ar[d]^{Hf} \\
  C' & (\Ho I)\widetilde QHC' \ar@{=}[r] & (\Ho I)H_0QC' \ar@{=}[r] & HIQC' \ar@{=}[r] & HQC' \ar[r]_-{Hq_{C'}} & HC' } \]
  for all $f\colon C \to C'$ in~$\Ceu$.
\end{proof}

\begin{remark}
  Because~$\Ho I$ has now been shown to be fully faithful (in fact, it is even an equivalence), we will usually view~$\Ho\Ceu_0$ as a full subcategory of~$\Ho\Ceu$ and under this identification~$H_0 = H\vert_{\Ceu_0}$.
\end{remark}

As already mentioned at the beginning of this section, the importance of deformation retracts comes from the fact that we can use them to ``deform'' functors in order to obtain derived ones.
\begin{definition}
  Let $H_\Ceu\colon \Ceu \to \Ho\Ceu$ and~$H_\Deu\colon \Deu \to \Ho\Deu$ be two localisation functors and~$F\colon \Ceu \to \Deu$. A left deformation retract~$\Ceu_0 \subseteq \Ceu$ is called a {\it left $F$\hyph deformation retract\/}\index{Retract!left deformation}\index{Deformation!retract}\index{Left!deformation retract} iff the restriction~$F\vert_{\Ceu_0}\colon \Ceu_0 \to \Deu$ is homotopical. Consequently, a left deformation retraction~$(\Ceu_0,Q,q)$ is called a {\it left $F$\hyph deformation retraction\/}\index{Retraction!left deformation}\index{Deformation!retraction}\index{Left!deformation retraction} iff~$\Ceu_0$ is a left $F$\hyph deformation retract. Dually for right deformation retracts\index{Retract!left deformation}\index{Left!deformation retract}\index{Retraction!right deformation}\index{Right!deformation retraction}.
\end{definition}

\begin{remark}\label{rem:functorial F-deformations}
  Given $F\colon \Ceu \to \Deu$ together with a small category~$\Ieu$, we always get an induced functor $F^\Ieu\colon \Ceu^\Ieu \to \Deu^\Ieu$ by applying~$F$ pointwise. Just as in \ref{rem:functorial deformations}, any left $F$\hyph deformation retraction~$(\Ceu_0,Q,q)$ where~$Q$ is actually a functor~$\Ceu \to \Ceu_0$ can be lifted to a left $F^\Ieu$\hyph deformation retraction by applying~$Q$ pointwise.
\end{remark}

\begin{theorem}
  Let~$(\Ceu_0,Q,q)$ be a left deformation retraction of\/~$\Ceu$ and consider a functor~$F\colon \Ceu \to \Deu$ that maps weak equivalences in~$\Ceu_0$ to isomorphisms. Then
  \[ LF\colon \Ho\Ceu \xrightarrow{\widetilde Q} \Ho\Ceu_0 \xrightarrow{\ol{F\vert_{\Ceu_0}}} \Deu \quad\text{together with}\quad (Fq_C\colon FQC \to FC)_{C \in \Ob\Ceu} \]
  is an absolute left derived functor of~$F$. In particular, if\/~$\Deu$ is also equipped with a class of weak equivalences and~$(\Ceu_0,Q,q)$ is a left $F$\hyph deformation retraction of~$\Ceu$, then
  \[ \Lbb F\colon \Ho\Ceu \xrightarrow{\widetilde Q} \Ho\Ceu_0 \xrightarrow{\Ho(F\vert_{\Ceu_0})} \Ho\Deu \quad\text{together with}\quad (H_\Deu Fq_C\colon FQC \to FC)_{C\in\Ob\Ceu} \]
  is an absolute total left derived functor of~$F$.
\end{theorem}
\begin{proof}
  Let us write~$\ol F \defas \ol{F\vert_{\Ceu_0}}\colon \Ho\Ceu_0 \to \Deu$ for the functor induced by~$F\vert_{\Ceu_0}$. We need to check that the~$Fq_C$ are natural in~$C \in \Ob\Ceu$; i.e.~that the square
  \begin{equation}\label{eqn:Fq nat} \vcenter{\hbox{$\xymatrix@C=5.5em{ FQC \ar[r]^-{\ol F\widetilde Q H f} \ar[d]_{Fq_C} & FQC' \ar[d]^{Fq_{C'}} \\ FC \ar[r]_-{Ff} & FC' }$}} \end{equation}
  in~$\Deu$ commutes for all morphisms~$f\colon C \to C'$ in~$\Ceu$. This is clear because we have $\ol F\widetilde Q Hf = \ol F HQf = FQf$. To see that this transformation is universal, let~$G\colon \Ho\Ceu \to \Deu$ and~$\tau\colon G\circ H \rar F$. If there is~$\ol\tau\colon G \rar LF$ such that~$Fq\circ\ol\tau_H = \tau$, then
  \[ GHC \xrightarrow{\ol\tau_{HC}} LFHC = FQC \xrightarrow{Fq_C} FC \quad=\quad GHC \xrightarrow{\tau_C} FC \qquad\text{for all }C \in \Ob\Ceu. \]
  For~$C \in \Ob{\Ceu_0}$, the arrow~$Fq_C$ is invertible and so $\ol\tau_{HC} = (Fq_C)^{-1}\tau_C$. For the general case, the naturality of~$\ol\tau$ gives a commutative diagram in~$\Deu$ as follows:
  \[ \xymatrix{GHQC \ar[d]_{GHq_C} \ar[r]^-{\ol\tau_{HQC}} & LFHQC \ar[d]^{LFHq_C} \ar@{=}[r] & FQQC \ar[d]^{FQq_C} \\ GHC \ar[r]_-{\ol\tau_{HC}} & LFHC \ar@{=}[r] & FQC & *!<1em,1ex>{.}} \]
  But~$GHq_C$ is invertible, so that
  \[ \bar\tau_{HC} = FQq_C \circ \bar\tau_{HQC} \circ (GHq_C)^{-1} = FQq_C \circ (Fq_{QC})^{-1} \circ \tau_{QC} \circ (GHq_C)^{-1}. \]
  We check that this does indeed define a transformation~$\bar\tau_H\colon G\circ H \rar LF\circ H$ (thus determining~$\bar\tau$) by considering the following commutative diagram for~$f\colon C \to C'$ in~$\Ceu$ (where for the last square, we note that while~$Q$ is not necessarily a functor, $FQ = \ol FH_0Q = \ol F\widetilde Q H$ certainly is):
  \[ \xymatrix@C=3em@R=1cm{ GHC \ar[r]^{(GHq_C)^{-1}} \ar[d]_{GHf} & GHQC \ar[r]^-{\tau_{QC}} \ar[d]_{GHQf} & FQC \ar[r]^{(Fq_{QC})^{-1}} \ar[d]_{FQf} & FQQC \ar[r]^-{FQq_C} \ar[d]^{FQQ f} & FQC \ar[d]^{FQf} \\ GHC' \ar[r]_-{(GHq_{C'})^{-1}} & GHQC' \ar[r]_-{\tau_{QC'}} & FQC' \ar[r]_-{(Fq_{QC'})^{-1}} & FQQC' \ar[r]_-{FQq_{C'}} & FQC' & *!<3em,1ex>{.} } \]
  Moreover, $\bar\tau$ does indeed satisfy~$Fq\circ\bar\tau_H = \tau$ because if~$C \in \Ob\Ceu$, then
  \begin{align*}
    Fq_C\circ\bar\tau_{HC}\ \ &=\ \ Fq_C \circ FQq_C \circ (Fq_{QC})^{-1} \circ \tau_{QC} \circ (GHq_C)^{-1} \\
      &\stackrel{\scriptscriptstyle\mathclap{Fq\text{ nat}}}{=}\ \ Fq_C \circ Fq_{QC} \circ (Fq_{QC})^{-1} \circ \tau_{QC} \circ (GHq_C)^{-1} \\
      &=\ \ Fq_C \circ \tau_{QC} \circ (GHq_C)^{-1} \\
      &\stackrel{\scriptscriptstyle\mathclap{\tau\text{ nat}}}{=}\ \ \tau_C \circ GHq_C \circ (GHq_C)^{-1}
  \end{align*}
  (for the second equality, put $C\leadsto QC$, $C' \leadsto C$, $f \leadsto q_C$ in \ref{eqn:Fq nat} above). Finally, for the absoluteness claim, observe that if~$F'\colon \Deu \to \Eeu$ is another functor, then~$F'\circ F$ again maps weak equivalences in~$\Ceu_0$ to isomorphisms.
\end{proof}

\begin{example}
  As already mentioned in \ref{ex:cofibrant objects deformation retract}, the full subcategory of cofibrant objects in a model category~$\Meu$, together with some chosen cofibrant replacements, is a left deformation retract of~$\Meu$. Consequently, a functor~$F\colon \Meu \to \Neu$ that sends weak equivalences between cofibrant objects to isomorphisms has an absolute left derived functor in the above manner. By Ken Brown's lemma, a left Quillen functor~$F\colon \Meu \to \Neu$ between two model categories (which preserves cofibrations and acyclic cofibrants) has an absolute total left derived functor.
\end{example}

\begin{example}
  More generally, in \cite{HToD}, a left model approximation $L\colon \Meu\! \rightleftarrows\! \Ceu\lon R$ is called {\it good\/} for a functor $F\colon \Ceu \to \Deu$ between two categories with weak equivalences iff $F\circ L$ sends weak equivalences between cofibrant objects to weak equivalences. This immediately implies that $\Ceu_0 \subseteq \Ceu$ as in \ref{ex:deformation retract for left model approximations} is a left $F$\hyph deformation retract and so we obtain $\Lbb F$.

  An important instance of this is that if~$\Meu$ is a model category, $F\colon \Ieu \to \Jeu$ a functor between small categories and~$\Neu(\Ieu) \defas \int_\Deltabf N(\Ieu)$ the category of simplices of the nerve of~$\Ieu$, there is the so-called Bousfield-Kan left model approximation
  \[ \Fun^b\bigl(\Neu(\Ieu),\Meu\bigr) \rightleftarrows \Fun(\Ieu,\Meu), \]
  which is good for the functor~$F_\bang$ given by taking the left Kan extension along~$F$. In particular, there is a homotopy left Kan extension functor $\Lbb F_!\colon \Ho(\Meu^\Ieu) \to \Ho(\Meu^\Jeu)$.
\end{example}

\begin{example}
  Consider a functor $F\colon \Ceu \to \Deu$, a small indexing category~$\Ieu$ and~$(\Ceu_0,Q,q)$ a left $F$\hyph deformation retraction such that~$Q\colon \Ceu \to \Ceu_0$ is even a functor. As mentioned in~\ref{rem:functorial F-deformations}, if $F^\Ieu\colon \Ceu^\Ieu \to \Deu^\Ieu$ is postcomposition with~$F$, then $(\Ceu_0^\Ieu,Q^\Ieu,q)$ is a left $F^\Ieu$\hyph deformation retraction (where $q_X\colon Q^\Ieu X = Q\circ X \rar X$ has components $q_{XI}$). In this case $\Lbb F$ is given by $\Ho(F\circ Q)$ and $\Lbb F^\Ieu$ is $\Ho(F^\Ieu\circ Q^\Ieu) = \Ho\bigl((F\circ Q)^\Ieu\bigr)$. More colloquially, the derived functor of~$F^\Jeu$ is just the derived functor of~$F$, applied pointwise.
\end{example}

To finish off this section, let us consider one instance where it is useful to have functorial factorisations in a model category. In that case, there is the following conceptual proof that the derived functors of left Quillen functors preserve homotopy colimits (even if the existence of projective model structures is not assumed). Here, we assume known that homotopy colimits in model categories can always be constructed, as is shown e.g.~in~\cite{HToD}.
\begin{proposition}\label{prop:quillen functors preserve hocolims}
  If $F\colon \Meu \rightleftarrows \Neu\lon G$ is a Quillen adjunction between two model categories with functorial factorisations and\/~$\Ieu$ is any small indexing category, there is a natural isomorphism
  \[ \Lbb F\circ\hocolim_\Ieu \cong {\hocolim_\Ieu}\circ\Lbb F^\Ieu. \]
\end{proposition}
\begin{proof}
  First off, even though~$\Meu^\Ieu$ and~$\Neu^\Ieu$ do not come with model structures, the derived functors~$\Lbb F^\Ieu$ and~$\Rbb G^\Ieu$ of $F^\Ieu\colon \Meu^\Ieu \to \Neu^\Ieu$ and $G^\Ieu\colon \Neu^\Ieu \to \Meu^\Ieu$ can still be constructed, as seen in the previous example. Writing~$\Delta$ for the constant diagram inclusion, to prove the claim, we can switch to adjoints and it suffices to prove that $\Ho\Delta \circ \Rbb G \cong \Rbb G^\Ieu \circ \Ho\Delta$. This is not hard because if~$R\colon \Neu \to \Neu_f$ is a fibrant replacement functor, $\Rbb G \cong \Ho(G\circ R)$ and $\Rbb G^\Ieu \cong \Ho(G^\Ieu\circ R^\Ieu)$.
\end{proof}

\section{The Yoga of Mates}
As is well-known, adjunctions are just a special instance of Kan extensions (then again, what isn't?) and the external characterisation makes it really obvious.
\begin{proposition}\label{prop:adjunctions and absolute Kan extensions}
  Let~$F\colon \Ceu \rightleftarrows \Deu\lon G$ be two functors. Then~$F\dashv G$ with counit~$\varepsilon$ iff~$(F,\varepsilon)$ is an absolute right Kan extension of\/~$\id_\Deu$ along~$G$ and dually, $F \dashv G$ with unit~$\eta$ iff~$(G,\eta)$ is an absolute left Kan extension of\/~$\id_\Ceu$ along~$F$. Indeed, if~$F \dashv G$ with unit~$\eta$ and counit~$\varepsilon$, then~$\tau \mapsto Y\varepsilon\circ\tau_G$ defines a bijection
  \[ \varphi_{\Eeu,X,Y}\colon \Nat(X,Y\circ F) \cong \Nat(X\circ G,Y) \qquad\text{with inverse}\qquad \sigma_F\circ X\eta \mapsfrom \sigma \]
  and this is natural in~$\Eeu$,~$X\colon \Ceu \to \Eeu$ and~$Y\colon \Deu \to \Eeu$. Conversely, given such a natural family of bijections, then~$F \dashv G$ with unit~$\varphi_{\Ceu,\id_\Ceu,G}^{-1}(\id_G)$ and counit~$\varphi_{\Deu,F,\id_\Deu}(\id_F)$. These two constructions are inverse to each other.
\end{proposition}
\begin{proof}
  The whole proof is simply about finding the correct naturality conditions to apply but let's do it anyway for the sake of completeness. Starting with an adjunction~$(F\dashv G,\eta,\varepsilon)$ and defining~$\varphi_{\Eeu,X,Y}$ as in the claim, the naturality of~$\varphi$ in~$\Eeu$ and~$X$ is immediate. For the naturality in~$Y$, we assume that we have $\alpha\colon Y \rar Y'$ and need to show that
  \[ \xymatrix@C=3em{ \Nat(X,Y\circ F) \ar[r]^{\varphi_{\Eeu,X,Y}} \ar[d]_{(\alpha_F)_*} & \Nat(X\circ G,Y) \ar[d]^{\alpha_*} \\ \Nat(X,Y'\circ F) \ar[r]_{\varphi_{\Eeu,X,Y'}} & \Nat(X\circ G, Y') } \]
  commutes. Chasing some~$\tau$ through the square, this means that
  \[ \alpha\circ Y\varepsilon\circ\tau_G = Y'\varepsilon\circ\alpha_{FG}\circ\tau_G, \]
  which follows from naturality of~$\alpha$. Moreover, the assignments
  \[ \tau \mapsto Y\varepsilon\circ\tau_G \qquad\text{and}\qquad \sigma_F\circ X\eta \mapsfrom \sigma \]
  are indeed inverse to each other, because, starting with~$\tau$, we have
  \[ \tau \longmapsto Y\varepsilon\circ\tau_G \longmapsto (Y\varepsilon\circ\tau_G)_F\circ X\eta \stackrel{\text{\tiny $\tau$ nat}}{=} Y\varepsilon_{FC}\circ YF\eta_C\circ \tau_C \stackrel{\text{\tiny $\Delta$-id}}{=} \tau_C \]
  and similarly the other way around.

  \vspace{.2ex}
  Conversely, starting with a natural family of~$\varphi_{\Eeu,X,Y}$, we put $\eta \defas \varphi^{-1}_{\Ceu,\id_\Ceu,G}(\id_G)$ and $\varepsilon \defas \varphi_{\Deu,F,\id_\Deu}(\id_F)$ as in the claim and need to verify the triangle identities. For example, to show $\varepsilon_F\circ F\eta = \id_F$ (the other one being similar), we just consider the commutative diagram
  \[ \xymatrix@C=3em{ \Nat(\id_\Ceu,G\circ F) \ar[r]^-{F_*} \ar[d]_{\varphi_{\Ceu,\id_\Ceu,G}} & \Nat(F,F\circ G\circ F) \ar[r]^-{(\varepsilon_F)_*} \ar[d]_{\varphi_{\Deu,F,FG}} & \Nat(F,F) \ar[d]_{\varphi_{\Deu,F,\id_\Deu}} \\ \Nat(G,G) \ar[r]_-{F_*} & \Nat(F\circ G,F\circ G) \ar[r]_-{\varepsilon_*} & \Nat(F\circ G,\id_\Deu), } \]
  where the left square commutes by naturality of $\varphi_{\Eeu,X,Y}$ in the variable~$\Eeu$, while the right square commutes by naturality in~$Y$. Chasing~$\eta = \varphi_{\Ceu,\id_\Ceu,G}^{-1}(\id_G)$ around the diagram, we find
  \[ \varphi_{\Deu,F,\id_\Deu}(\varepsilon_F\circ F\eta) = \varepsilon = \varphi_{\Deu,F,\id_\Deu}(\id_F) \]
  and the claim follows. 

  \vspace{.2ex}
  Finally, the two assignments $(\eta,\varepsilon) \mapsto \varphi$ and $(\eta,\varepsilon) \mapsfrom \varphi$ as in the claim are mutually inverse, because, starting from~$(\eta,\varepsilon)$, constructing $\varphi_{\Eeu,X,Y}\colon \tau \mapsto Y\varepsilon\circ\tau_G$ and taking the associated unit and counit, the new counit is
  \[ \varphi_{\Deu,F,\id_\Deu}(\id_F) = \id_\Deu\varepsilon\circ(\id_F)_G = \varepsilon, \]
  which also shows that the new unit is again~$\eta$ because the counit determines the unit and vice versa. Conversely, starting with~$\varphi$, and taking the associated unit $\eta \defas \varphi^{-1}_{\Ceu,\id_\Ceu,G}(\id_G)$ and counit $\varepsilon \defas \varphi_{\Deu,F,\id_\Deu}(\id_F)$, the associated natural family~$\varphi'$ is defined as
  \[ \varphi'_{\Eeu,X,Y}\colon \tau \mapsto Y\varepsilon\circ\tau_G = Y\varphi_{\Deu,F,\id_\Deu}(\id_F)\circ\tau_G. \]
  Now, $\tau$ is natural transformation $X \rar Y\circ F$ and by naturality of~$\varphi_{\Eeu,X,Y}$ in the variable~$X$, the square
  \[ \xymatrix@C=4em{ \Nat(Y\circ F,Y\circ F) \ar[r]^{\varphi_{\Eeu,Y\circ F,Y}} \ar[d]_{\tau^*} & \Nat(Y\circ F\circ G,Y) \ar[d]^{(\tau_G)^*} \\ \Nat(X,Y\circ F) \ar[r]_{\varphi_{\Eeu,X,Y}} & \Nat(X\circ G,Y) } \]
  commutes. Chasing~$\id_{Y\circ F}$ through the square, we find that
  \[ \varphi'_{\Eeu,X,Y}(\tau) = \varphi_{\Eeu,X,Y}(\tau\circ\id_{Y\circ F}) = \varphi_{\Eeu,X,Y}(\tau). \]
\end{proof}

As always, by dualising (really {\it op\/}-dualising, i.e.~inverting 1\hyph arrows), we get an alternative external characterisation of adjunctions, which exhibits adjoints as {\it Kan lifts\/}\index{Kan lift} and also fixes the counter-intuitive aspect that in the above formula, the left adjoint appears on the right. 
\begin{proposition}
  Given two functors~$F\colon \Ceu \rightleftarrows \Deu\lon G$ we have~$F \dashv G$ with unit~$\eta$ and counit~$\varepsilon$ iff~$\sigma \mapsto G\sigma\circ\eta_X = \sigma^\sharp$ defines a bijection
\[ \Nat(F\circ X,Y) \cong \Nat(X,G\circ Y) \qquad\text{with inverse}\qquad \tau^\flat = \varepsilon_Y\circ F\tau \mapsfrom \tau, \]
which is natural in $\Eeu$,~$X\colon \Eeu \to \Ceu$ and~$Y\colon \Eeu \to \Deu$. Conversely, given such a natural family of bijections, then~$F \dashv G$ with unit~$\psi_{\Ceu,\id_\Ceu,F}(\id_F)$ and counit~$\psi_{\Deu,G,\id_\Deu}^{-1}(\id_G)$. These two constructions are inverse to each other.
\end{proposition}
Combining the two bijections from these propositions leads to the well-known {\it mating-bijection\/}\index{Bijection!mating}\index{Mating bijection}
\begin{equation}\label{eqn:mating bijection}
  \vcenter{\hbox{$\begin{array}{rcl} \sigma &\mapsto& \sigma^{\mate}\\ \Nat(F'\circ X, Y\circ F) &\cong& \Nat(X\circ G, G'\circ Y)\\ {}^{\mate}\tau &\mapsfrom& \tau \end{array}$}}\end{equation}
\[ \vcenter{\hbox{for categories and functors}}\quad
  \vcenter{\hbox{$\xymatrix{ \Ceu \ar@<1ex>[r]^-{F}^-{}="1" \ar[d]_{X} & \Deu \ar@<1ex>[l]^-{G}^-{}="2" \ar[d]^{Y} \\ \Ceu' \ar@<1ex>[r]^-{F'}^-{}="11" & \ar@<1ex>[l]^-{G'}^-{}="12" \Deu' & *!<2em,1ex>{.} \ar@{}"1";"2"|(.3){\hbox{}}="3"|(.7){\hbox{}}="4" \ar@{|-}"4";"3" \ar@{}"11";"12"|(.3){\hbox{}}="13"|(.7){\hbox{}}="14" \ar@{|-}"14";"13" }$}} \]
This bijection is natural in~$X$ and~$Y$ and is explicitly given by
  \[ \sigma \mapsto G'Y\varepsilon\circ G'\sigma_G\circ\eta'_{XG} = (Y\varepsilon\circ\sigma_G)^\sharp, \qquad (\tau_F\circ X\eta)^\flat = \varepsilon'_{YF}\circ F'\tau_F\circ F'X\eta \mapsfrom \tau. \]

\begin{definition}
  We write $\sigma\mate\tau$ and say~$\sigma\colon F'X \rar YF$ and~$\tau\colon XG \rar G'Y$ are {\it mates\/}\index{Mates}\index{Natural!transformation!mate}\index{Transformation!mate} (we should really say that~$\sigma = {}^{\mate}\tau$ is the {\it left mate\/} and~$\tau = \sigma^{\mate}$ is the {\it right mate\/}) iff they correspond to each other under this mating bijection. For the special case where~$X$ and~$Y$ are identities, we will occasionally use Mac Lane's nomenclature from \cite{CatsWork} and speak of {\it conjugate\/}\index{Natural transformation!conjugate}\index{Conjugate!natural transformations} transformations.
\end{definition}

\begin{remark}\label{rem:naturality of the mating bijection}
  A useful consequence of the mating bijection's naturality in~$X$ and~$Y$ is that for $\alpha\colon X \rar X'$,~$\beta\colon Y \rar Y'$ and two squares of natural transformations
  \[ \xymatrix{ F'\circ X \ar[r]^-{\sigma} \ar[d]_{F\alpha} & Y\circ F \ar[d]^{\beta_F} \\ F'\circ X' \ar[r]_-{\sigma'} & Y'\circ F } \qquad \xymatrix{ X\circ G \ar[r]^-{\tau} \ar[d]_{\alpha_G} & G'\circ Y \ar[d]^{G'\beta} \\ X'\circ G \ar[r]_-{\tau'} & G'\circ Y' } \]
  where the two horizontal pairs are mates, $\beta_F\circ\sigma$ and $G'\beta\circ\tau$ as well as $\sigma'\circ F\alpha$ and $\tau'\circ\alpha_G$ are again mates. Consequently, the left-hand square commutes iff the right-hand square commutes.
\end{remark}
  
  Due to a lack of references, we shall briefly establish some standard results about mates, which one expects to be true and whose proofs are purely formal (whence the term {\it yoga\/}).
\begin{example}\label{ex:mates and endofunctors}
  If~$X$ and~$Y$ are identities and the two adjunctions $F \dashv G$,~$F'\dashv G'$ are the same, then~$\id_F\mate\id_G$, which is just a complicated way to state the triangle identities. More generally, if the two adjunctions are the same, while~$X$ and~$Y$ are endofunctors equipped with~$\alpha\colon X \rar \id_\Ceu$ and~$\beta\colon \id_\Deu \rar Y$, then by naturality of the mating bijection, $\beta_F\circ F\alpha\mate G\beta\circ\alpha_G$. Dually for the directions of~$\alpha$ and~$\beta$ reversed.
\end{example}

\begin{proposition}
  In the same situation as in the definition, for two natural transformations~$\sigma\colon F'X \rar YF$ and~$\tau\colon XG \rar G'Y$, the following are equivalent:
  \begin{enumerate}
    \item $\sigma\mate\tau$, i.e.~$\sigma = (\tau_F\circ X\eta)^\flat$ or equivalently $\tau = (Y\varepsilon\circ\sigma_G)^\sharp$;
    \item $\sigma^\sharp = G'\sigma\circ\eta'_X = \tau_F\circ X\eta$ or equivalently~$\tau^\flat = \varepsilon'_Y\circ F'\tau = Y\varepsilon\circ\sigma_G$;
    \item for all~$C \in \Ob\Ceu$ and all~$D \in \Ob\Deu$ the rectangle
      \[ \xymatrix{\Deu(FC,D) \ar[r]^{\cong} \ar[d]_{Y} & \Ceu(C,GD) \ar[d]^{X} \\ \Deu'(YFC,YD) \ar[d]_{\sigma_C^*} & \Ceu'(XC,XGD) \ar[d]^{(\tau_D)_*} \\ \Deu'(F'XC,YD) \ar[r]_{\cong} & \Ceu'(XC,G'YD)} \]
      commutes, where the horizontal arrows are the tuning bijections of the two adjunctions.
  \end{enumerate}
\end{proposition}
\begin{proof}
  ``(1) $\Leftrightarrow$ (2)'': Trivial.\\
  ``(2) $\Leftrightarrow$ (3)'': Condition (2) is just the commutativity of the diagram for~$D = FC$ and~$C = GD$ respectively. Conversely, this implies the commutativity for all~$C$ and~$D$ by naturality of the tuning bijection for the adjunction.
\end{proof}

\begin{proposition}\label{prop:composition of mates}
  Taking mates is compatible with vertical and horizontal pasting of squares in the following sense:
  \begin{enumerate}
    \item  Given categories and functors
      \[ \xymatrix{ \Ceu \ar[r]^-{X} \ar@<-1ex>[d]_{F}_{}="1" & \Ceu' \ar[r]^-{X'} \ar@<-1ex>[d]_{F'}_{}="11" & \Ceu'' \ar@<-1ex>[d]_{F''}_{}="21" \\
      \Deu \ar[r]_-{Y} \ar@<-1ex>[u]_{G}_{}="2" & \Deu' \ar[r]_-{Y'} \ar@<-1ex>[u]_{G'}_{}="12" & \Deu'' \ar@<-1ex>[u]_{G''}_{}="22" 
        \ar@{}"1";"2"|(.3){\hbox{}}="3"|(.7){\hbox{}}="4" \ar@{-|}"3";"4"
        \ar@{}"11";"12"|(.3){\hbox{}}="13"|(.7){\hbox{}}="14" \ar@{-|}"13";"14"
        \ar@{}"21";"22"|(.3){\hbox{}}="23"|(.7){\hbox{}}="24" \ar@{-|}"23";"24"
      } \]
      together with transformations~$\sigma\colon F'X \rar YF$ and~$\sigma'\colon F''X' \rar Y'F'$ having mates~$\tau = \sigma^{\mate}\colon XG \rar G'Y$ and~$\tau'=\sigma'^{\mate}\colon X'G' \rar G''Y'$ respectively, then~$Y'\sigma\circ\sigma'_X$ and~$\tau'_Y\circ X'\tau$ are again mates. In particular, if $X$, $X'$,~$Y$ and~$Y'$ are identities, then~$\sigma\circ\sigma'$ and~$\tau'\circ\tau$ are conjugate.
    \item Given categories and functors
      \[ \xymatrix@R=6ex{ \Ceu \ar[d]_{X} \ar@<1ex>[r]^-{F_1}^-{}="1" & \Deu \ar[d]_{Y} \ar@<1ex>[l]^-{G_1}^-{}="2" \ar@<1ex>[r]^-{F_2}^-{}="11" & \Eeu \ar[d]^{Z} \ar@<1ex>[l]^-{G_2}^-{}="12" \\ \Ceu' \ar@<1ex>[r]^-{F_1'}^-{}="21" & \Deu' \ar@<1ex>[l]^-{G_1'}^-{}="22" \ar@<1ex>[r]^-{F_2'}^-{}="31" & \Eeu' \ar@<1ex>[l]^-{G_2'}^-{}="32"
      \ar@{}"1";"2"|(.3){\hbox{}}="3"|(.7){\hbox{}}="4" \ar@{-|}"3";"4"
      \ar@{}"11";"12"|(.3){\hbox{}}="13"|(.7){\hbox{}}="14" \ar@{-|}"13";"14"
      \ar@{}"21";"22"|(.3){\hbox{}}="23"|(.7){\hbox{}}="24" \ar@{-|}"23";"24"
      \ar@{}"31";"32"|(.3){\hbox{}}="33"|(.7){\hbox{}}="34" \ar@{-|}"33";"34"
      } \]
      together with transformations~$\sigma_1\colon F_1'X \rar YF_1$ and~$\sigma_2\colon F_2'Y \rar ZF_2$ having mates~$\tau_1\colon XG_1 \rar G_1'Y$ and~$\tau_2\colon YG_2 \rar G_2'Z$, then~${\sigma_2}_{F_1}\circ F_2'\sigma_1$ and~$G_1'\tau_2\circ{\tau_1}_{G_2}$ are again mates. In particular, if $X$,~$Y$ and~$Z$ are identities, then the horizontal composites~$\sigma'\star\sigma$ and~$\tau\star\tau'$ are conjugate. Still more specially, if in addition~$F_2 = F_2'$ and~$\sigma_2$ (whence~$\tau_2$) is the identity, we obtain that~$F_2\sigma_1$ and~${\tau_1}_{G_2}$ are conjugate.
  \end{enumerate}
\end{proposition}
\begin{proof}
  {\it Ad\/} (1): Writing $(\eta,\varepsilon)$,~$(\eta',\varepsilon')$ and~$(\eta'',\varepsilon'')$ for the unit\hyph counit pairs of the three adjunctions we easily calculate
  \begin{align*}
    \bigl(Y'\sigma\circ\sigma'_X\bigr)^\sharp &= G''Y'\sigma\circ G''\sigma'_X\circ\eta''_{X'X} = G''Y'\sigma\circ{\sigma'}^\sharp_X \\
      &= G''Y'\sigma\circ\tau'_{F'X}\circ X'\eta'_X \stackrel{\text{\tiny $\tau'$ nat}}{=} \tau'_{YF}\circ X'G'\sigma\circ X'\eta_X \\
      &= \tau'_{YF}\circ X'\sigma^\sharp = \tau'_{YF}\circ X'\tau_F\circ X'X\eta.
  \end{align*}
  {\it Ad\/} (2): Again, writing $(\eta_1,\varepsilon_1)$, $(\eta_2,\varepsilon_2)$~$(\eta_1',\varepsilon_1')$ and~$(\eta_2',\varepsilon_2')$ for the unit\hyph counit pairs this is just a routine calculation:
    \begin{align*}
      \bigl({\sigma_2}_{F_1}\circ F_2'\sigma_1\bigr)^\sharp\ \ &=\ \ G_1'G_2'{\sigma_2}_{F_1}\circ G_1'G_2'F_2'{\sigma_1}\circ (G_1'{\eta_2'}_{F_1'}\circ\eta_1')_X \\
        &\stackrel{\scriptscriptstyle\mathclap{\eta_2'\text{ nat}}}{=}\ \ G_1'G_2'{\sigma_2}_{F_1}\circ G_1'{\eta_2'}_{YF_1}\circ G_1'\sigma_1\circ{\eta_1'}_X  \\[1ex]
        &=\ \ G_1'{\sigma_2^\sharp}_{F_1}\circ\sigma_1^\sharp = G_1'{\tau_2}_{F_2F_1}\circ G_1'Y{\eta_2}_{F_1}\circ {\tau_1}_{F_1}\circ X\eta_1 \\[1ex]
        &\stackrel{\scriptscriptstyle\mathclap{\tau_1\text{ nat}}}{=}\ \ G_1'{\tau_2}_{F_2F_1}\circ{\tau_1}_{G_2F_2F_1}\circ XG_1{\eta_2}_{F_1}\circ X\eta_1 \\[1ex]
        &=\ \ \bigl(G_1'\tau_2\circ{\tau_1}_{G_1}\bigr)_{F_2F_1}\circ X\bigl(G_1{\eta_2}_{F_1}\circ\eta_1\bigr).
     \end{align*}
\end{proof}

\begin{corollary}\label{cor:conjugate isomorphisms}
  Given two adjunctions $(F\dashv G,\eta,\varepsilon)$, $(F'\dashv G',\eta',\varepsilon')\colon \Ceu \rightleftarrows\Deu$ and conjugate transformations $\sigma\colon F \rar F'$,~$\tau\colon G' \rar G$, then~$\sigma$ is an isomorphism iff~$\tau$ is one. 
\end{corollary}
\begin{proof}
  Let $\tau'\colon G' \rar G$ be the mate of~$\sigma^{-1}$. Then by point (1) in the proposition (with $F'' = F$, $G'' = G$ and $X$, $X'$, $Y$,~$Y'$ all identities) $\sigma\circ\sigma^{-1} = \id_F$ and~$\tau'\circ\tau$ are mates, so that $\tau'\circ\tau = \id_{G'}$. Similarly, $\sigma^{-1}\circ\sigma = \id_{F'}$ and~$\tau\circ\tau'$ are mates, which proves our claim.
\end{proof}

\section{Beck-Chevalley Condition}
Later on, we will find ourselves in the situation where we have a mating square as in \ref{eqn:mating bijection} with a natural isomorphism $\tau\colon XG\cong G'Y$ but where we would really like its mate $\sigma\colon F'X \rar YF$ to be an isomorphism.
\begin{definition}\label{defn:beck-chevalley}
  Consider a square of categories and functors together with a natural transformation
  \[ \vcenter{\xymatrix{ \Ceu \ar[d]_{X} \ar@{}[dr]|(.4){}="1"|(.6){}="2" & \Deu \ar[l]_{G} \ar[d]^{Y} \\ \Ceu' & \Deu' \ar[l]^{G'} \ar@{=>}"1";"2"^{\tau} }} \qquad
  \newbox\bcdefbox\setbox\bcdefbox=\hbox{where~$G$ and~$G'$ have left adjoints}
  \vcenter{\vbox{\copy\bcdefbox\vskip.5\baselineskip\hbox to \wd\bcdefbox{\hfil$(F\dashv G,\eta,\varepsilon),\, (F'\dashv G',\eta',\varepsilon')$.\hfil}}} \]
  We then say that the square satisfies the {\it Beck-Chevalley condition\/}\index{Beck-Chevalley!condition} (or that it's a {\it Beck-Chevalley square\/}\index{Beck-Chevalley!square}\index{Square!Beck-Chevalley}) iff the mate~${}^{\mate}\tau$ is an isomorphism~$F'\circ X \cong Y\circ F$. Dually, there is the {\it dual Beck-Chevalley condition\/}, where we start with $F$,~$F'$ and~$\sigma$ and then require the mate~$\sigma^{\mate}$ to be an isomorphism. Usually, $\tau$ is some sort of canonical isomorphism and is then often not explicitly mentioned. However, possible confusion can arise if the functors~$X$ and~$Y$ themselves have left adjoints. In that case, we shall speak of the {\it horizontal\/} and {\it vertical Beck-Chevalley condition\/} according to whether one considers the horizontal or vertical pairs of adjunctions; the case in the above definition being the horizontal one.
\end{definition}

\begin{example}
  By \ref{cor:conjugate isomorphisms}, if~$X$ and~$Y$ are identities, then the square from the definition satisfies the Beck-Chevalley condition iff~$\tau$ is an isomorphism.
\end{example}

\begin{example}
  The mate of~$\tau$ is explicitly given by $\varepsilon'_{YF}\circ F'\tau_F\circ F'X\eta$; so if~$\tau$ is an isomorphism and $F$,~$G'$ are fully faithful (i.e.~$\eta$ and~$\varepsilon'$ are isomorphisms), then the square satisfies the Beck-Chevalley condition.
\end{example}

\begin{example}\label{ex:beck-chevalley for colimits}
  More importantly for us, if $\Ieu$,~$\Jeu$ are index categories and~$\Ceu$ is a category with $\Ieu$\hyph colimits, then $\Ceu^\Jeu$ has $\Ieu$\hyph colimits, too. A colimit functor is given by
  \[ \colim\colon \Ceu^{\Ieu\times\Jeu}\cong(\Ceu^\Jeu)^\Ieu \to \Ceu^\Jeu,\, X \mapsto \Bigl(J \mapsto \colim X(-,J)\Bigr) \]
  (i.e.~colimits are calulcated pointwise). If we write~$\eta'$ for the unit of the adjunction $\colim\colon \Ceu^\Ieu \rightleftarrows\Ceu\lon \Delta$, whose components are just the universal cocones, then a unit of $\colim\colon \Ceu^{\Ieu\times\Jeu}\rightleftarrows \Ceu^\Jeu\lon\Delta$ is given by $\eta_{X,I,J} \defas \eta'_{X(-,J),I}$. All in all, the square
  \[ \xymatrix{ \Ceu^{\Ieu\times\Jeu}\ar[d]_{\ev_J}\ar@{}[dr]|(.4){}="1"|(.6){}="2" & \Ceu^\Jeu \ar[l]_-{\Delta}\ar[d]^{\ev_J} \\ \Ceu^\Ieu & \Ceu \ar[l]^-{\Delta} \ar@{=>}"1";"2"^{\id} } \]
  satisfies the Beck-Chevalley condition for all $J \in \Ob{\Jeu}$. Indeed, the mate of the identity is again the identity and $\ev_J\eta = \eta'_{\ev_J}$.
\end{example}

\begin{observation}
  According to~\ref{prop:composition of mates}, horizontal and vertical composites of Beck-Che\-va\-lley squares (defined in the obvious manner) are again Beck-Chevalley squares.
\end{observation}

The situation gets really interesting when~$X$ and~$Y$ do have right adjoints. In that case, there is the following interchange law for mates as stated e.g.~in \cite[Lemma 1.20]{Groth2013}.%
\begin{theorem}[Beck-Chevalley Interchange]\hskip1em
  Consider a square of categories and functors together with a natural transformation $\tau\colon XG \rar G'Y$ as in the above definition and where all four functors have adjoints
  \[ (F \dashv G,\eta,\varepsilon),\quad (F'\dashv G',\eta',\varepsilon'),\quad (X \dashv S,\theta,\zeta)\quad\text{and}\quad (Y \dashv T,\theta',\zeta'), \]
  so that~$\tau$ has a horizontal left mate $\sigma\!\colon\! F'X \rar YF$ as well as a vertical right mate $\rho\!\colon\! GT \rar SG'$. Then~$\sigma$ and~$\rho$ are conjugate natural transformations. In particular, $\sigma$ is an isomorphism iff~$\rho$ is one and thus the square satisfies the horizontal Beck-Chevalley condition iff it satisfies the vertical dual Beck-Chevalley condition.
\end{theorem}
\begin{proof}
  Recall that the two mates~$\sigma$ and~$\rho$ are defined by
  \[ G'\sigma\circ\eta'_X = \tau_F\circ X\eta \qquad\text{and}\qquad S\tau\circ\theta_G = \rho_Y\circ G\theta' \]
  and we need to check that~$\sigma$ and~$\rho$ are conjugate with respect to the composite adjunctions
  \[ \xymatrix{ \Ceu \ar@<1ex>[r]^-{F}^-{}="1"\ar@{=}[d] & \Deu \ar@<1ex>[l]^-{G}^-{}="2" \ar@<1ex>[r]^-{Y}^-{}="11" & \Deu' \ar@<1ex>[l]^-{T}^-{}="12" \ar@{=}[d] \\  \Ceu \ar@<1ex>[r]^-{X}^-{}="21" & \Ceu' \ar@<1ex>[l]^-{S}^-{}="22" \ar@<1ex>[r]^-{F'}^-{}="31" & \Deu' \ar@<1ex>[l]^-{G'}^-{}="32" & *!<2em,.5ex>{.}
      \ar@{}"1";"2"|(.3){\hbox{}}="3"|(.7){\hbox{}}="4" \ar@{-|}"3";"4"
      \ar@{}"11";"12"|(.3){\hbox{}}="13"|(.7){\hbox{}}="14" \ar@{-|}"13";"14"
      \ar@{}"21";"22"|(.3){\hbox{}}="23"|(.7){\hbox{}}="24" \ar@{-|}"23";"24"
      \ar@{}"31";"32"|(.3){\hbox{}}="33"|(.7){\hbox{}}="34" \ar@{-|}"33";"34" } \]
   For this, we just need to take the adjunct of~$\sigma$, which is
   \begin{align*}
     \sigma^\sharp &= S(G'\sigma\circ\eta'_X)\circ\theta = S(\tau_F\circ X\eta)\circ\theta = S\tau_F\circ SX\eta\circ\theta \\
      &= S\tau_F\circ\theta_{GF}\circ\eta = (S\tau\circ\theta_G)_F\circ\eta = (\rho_Y\circ G\theta')_F\circ\eta = \rho_{YF}\circ G\theta'_F\circ\eta
   \end{align*}
 and~$G\theta'_F\circ\eta$ is the unit of the upper composite adjunction.
\end{proof}

\begin{corollary}
  Consider a square as in~\ref{defn:beck-chevalley} with~$\tau\colon XG\cong G'Y$ an isomorphism. If~$Y$ has a fully faithful right adjoint and~$X$ is itself fully faithful and has a right adjoint, then the square satisfies the Beck-Chevalley condition.
\end{corollary}
\begin{proof}
  The square satisfies the (horizontal) Beck-Chevalley condition iff it satisfies the vertical dual Beck-Chevalley condition. Taking adjoints
  \[ (X\dashv S,\theta,\zeta)\colon \adj{\Ceu}{\Ceu'} \qquad\text{and}\qquad (Y\dashv T,\theta',\zeta')\colon \adj{\Deu}{\Deu'} \]
  the mate of~$\tau$ for the vertical adjoints is~$SG'\zeta'\circ S\tau_T\circ\theta_{GT}$.  But~$\tau$ is an isomorphism, same as~$\zeta'$ (since~$T$ is fully faithful) and~$\theta$ (since~$X$ is fully faithful).
\end{proof}

\begin{corollary}
  Given a mating square with~$X$ and~$Y$ equivalences
  \[ \vcenter{\xymatrix{ \Ceu \ar@<1ex>[r]^-{F}^-{}="1" \ar[d]_{X} & \Deu \ar@<1ex>[l]^-{G}^-{}="2" \ar[d]^{Y} \\ \Ceu' \ar@<1ex>[r]^-{F'}^-{}="11" & \ar@<1ex>[l]^-{G'}^-{}="12" \Deu' & *!<2em,.5em>{,}
     \ar@{}"1";"2"|(.3){\hbox{}}="3"|(.7){\hbox{}}="4" \ar@{|-}"4";"3"
     \ar@{}"11";"12"|(.3){\hbox{}}="13"|(.7){\hbox{}}="14" \ar@{|-}"14";"13" }} \]
  then~$\sigma\colon F'X \rar YF$ is an isomorphism iff its mate~$\tau\colon XG \rar G'Y$ is so.
\end{corollary}
\begin{proof}
  Every equivalence is fully faithful and has a fully faithful left and right adjoint.
\end{proof}

Unfortunately, if the two vertical arrows~$X$ and~$Y$ in \ref{defn:beck-chevalley} have left adjoints rather than right ones, there is no nice interchange law. However, we can still use such adjoints to our advantage.
\begin{theorem}
  Again consider a square as in~\ref{defn:beck-chevalley}
  with~$\tau\colon XG\cong G'Y$ a natural isomorphism and where all four functors have left adjoints
  \[ (F \dashv G,\eta,\varepsilon),\quad (F'\dashv G',\eta',\varepsilon'),\quad (M \dashv X,\theta,\zeta)\quad\text{and}\quad (N \dashv Y,\theta',\zeta'). \]
  \begin{enumerate}
    \item If~$X$ and~$N$ are fully faithful, then the square satisfies the Beck-Chevalley condition.
    \item If both~$M$ and~$N$ or both~$X$ and~$Y$ are fully faithful, then the square satisfies the Beck-Chevalley condition iff there is an isomorphism $F'X \cong YF$.
  \end{enumerate}
\end{theorem}
\begin{proof}
  The mate~$\sigma$ of~$\tau$ is the unique natural transformation that satisfies $G'\sigma\circ\eta'_X = \tau_F\circ X\eta$ and so, it suffices to construct a natural isomorphism~$\sigma$ subject to this equation. For this, we consider the conjugate of~$\tau$ with respect to the two composite adjunctions, which is the unique isomorphism
  \[ \rho\colon N \circ F' \cong F\circ M\qquad\text{satisfying}\qquad G'Y\rho\circ G'\theta'_{F'}\circ\eta' = \tau_{FM}\circ X\eta_M\circ\theta. \]
  Evaluating this equation at~$X$ and postcomposing with~$G'YF\zeta$ yields
\begin{align*}
  \MoveEqLeft G'(YF\zeta\circ Y\rho_X\circ \theta'_{F'X})\circ\eta'_X = G'YF\zeta\circ \tau_{FMX}\circ X\eta_{MX}\circ\theta_X \\
    &= \tau_F\circ XGF\zeta\circ X\eta_{MX}\circ\theta_X = \tau_F\circ X\eta\circ X\zeta\circ\theta_X = \tau_F\circ X\eta,
\end{align*}
  so that the mate of~$\tau$ is $\sigma = YF\zeta\circ Y\rho_X\circ\theta'_{F'X}$. Under the hypotheses of (1), this is an isomorphism since~$\zeta$ and~$\theta'$ are invertible, thus proving the first point. For point (2), assume that~$M$ and~$N$ are fully faithful and there is an isomorphism $F'X \cong YF$. We need to check that~$YF\zeta$ is invertible, which is obvious since $YF\cong F'X$ and~$X\zeta = \theta_X^{-1}$ is invertible. Similarly if~$X$ and~$Y$ are fully faithful.
\end{proof}

\section{Derived Adjunctions}\label{sec:derived adjunctions}
As already mentioned, there is a beauti- and useful interplay between adjunctions and absolute derived functors; the most well-known instance of it perhaps being the famous Quillen adjoint functor theorem. Most proofs of it rely heavily on the explicit construction of a derived functor by means of (co)fibrant replacements whereas our approach really gets down to its bare bones. For this, let us fix two categories $\Ceu$,~$\Deu$ each equipped with a class of weak equivalences and let us write~$H_\Ceu$ and~$H_\Deu$ for the corresponding localisations.

The results proven in this section have been obtained independently by González \cite{Gonzalez2012} though using very different, more diagrammatic methods.
\begin{theorem}\label{thm:adjoint gives derived functor}
  Let~$(F\dashv G,\eta,\varepsilon)\colon \Ceu \rightleftarrows \Deu$ be an adjunction such that~$G$ has an absolute total right derived functor~$(\Rbb G,\rho)$. If this in turn has a left adjoint~$(\dot F\dashv \Rbb G,\dot\eta,\dot\varepsilon)$, then~$(\dot F,\lambda)$ is an absolute total left derived functor of~$F$, where~$\lambda \defas \dot\varepsilon_{H_\Deu F}\circ \dot F\rho_F\circ \dot FH_\Ceu\eta$.
\end{theorem}
\begin{proof}
  We need a bijection~$\Nat(L,E\circ F') \cong \Nat(L\circ H_\Ceu,E\circ H_\Deu\circ F)$, natural in~$\Eeu$,~$E\colon \Ho\Deu \to \Eeu$ and~$L\colon \Ho\Ceu \to \Eeu$ and can construct one such by
  \setlength\oldarraycolsep\arraycolsep%
  \setlength\arraycolsep{1.4pt}%
  \[\begin{array}{rcl@{\qquad}l}
    \Nat(L, E\circ \dot F) &\cong& \Nat(L\circ\Rbb G, E) & \tau \mapsto E\dot\varepsilon\circ\tau_{\Rbb G} \\
      &\cong& \Nat(L\circ H_\Ceu\circ G, E\circ H_\Deu) & \tau \mapsto \tau_{H_\Deu}\circ L\rho \\
      &\cong& \Nat(L\circ H_\Ceu, E\circ H_\Deu\circ F) & \tau \mapsto \tau_F\circ (L\circ H_\Ceu)\eta.
    \end{array}\]
  \setlength\arraycolsep\oldarraycolsep
  Upon putting $L \defas \dot F$,~$E \defas \id_{\Ho\Deu}$ and chasing~$\id_{\dot F}$ through the bijections, the counit~$\lambda$ has indeed the claimed form.
\end{proof}

\begin{theorem}\label{thm:derived functors are adjoint}
  Let~$(F\dashv G,\eta,\varepsilon)\colon \Ceu \rightleftarrows \Deu$. If~$F$ has an absolute total left derived functor~$(\Lbb F,\lambda)$ and~$G$ has an absolute total right derived functor~$(\Rbb G,\rho)$, then we get an adjunction~$(\Lbb F \dashv \Rbb G,\dot\eta,\dot\varepsilon)$ where~$\dot\eta$ is the unique~$\id_{\Ho\Ceu} \rar \Rbb G\circ\Lbb F$ such that\/~$\Rbb G\lambda\circ\dot\eta_{H_\Ceu} = \rho_F\circ H_\Ceu\eta$ and~$\dot\varepsilon$ is the unique~$\Lbb F\circ\Rbb G \rar \id_{\Ho\Deu}$ such that~$H_\Deu\varepsilon\circ\lambda_G = \dot\varepsilon_{H_\Deu}\circ\Lbb F\rho$.
\end{theorem}
\begin{proof}
  We need to construct a bijection~$\Nat(E'\circ\Rbb G, E) \cong \Nat(E', E\circ\Lbb F)$ natural in~$\Eeu$, $E'\colon \Ho\Ceu \to \Eeu$ and~$E\colon \Ho\Deu \to \Eeu$. We can do so by
  \setlength\oldarraycolsep\arraycolsep%
  \setlength\arraycolsep{1.4pt}%
  \[\begin{array}{rcl@{\qquad}l@{\quad}l}
    \Nat(E'\circ\Rbb G,E) &\cong& \Nat(E'\circ H_\Ceu\circ G, E\circ H_\Deu) & \sigma \mapsto \sigma_{H_\Deu}\circ E'\rho \\
        &\cong& \Nat(E'\circ H_\Ceu, E\circ H_\Deu\circ F) & \sigma \mapsto \sigma_F\circ E'H_\Ceu\eta, \\
        &&& EH_\Deu\varepsilon\circ\tau_G \mapsfrom\tau \\
        &\cong& \Nat(E', E\circ\Lbb F) & E\lambda\circ\tau_{H_\Ceu} \mapsfrom \tau
    \end{array}\]
  \setlength\arraycolsep\oldarraycolsep
  and it plainly follows that~$\dot\eta$ and~$\dot\varepsilon$ are of the required form.
\end{proof}

The explicit descriptions of the unit and counit in the last theorem is not very enlightening and it might be clearer (although the author is not convinced) to draw the corresponding diagrams:%
\begin{equation}\label{eqn:compatibility of units and counits} \vcenter{ \xymatrix{ H_\Ceu \ar[r]^-{H_\Ceu\eta} \ar[d]_{\dot\eta_{H_\Ceu}} & H_\Ceu\circ G\circ F \ar[d]^{\rho_F} \\ \Rbb G\circ\Lbb F\circ H_\Ceu \ar[r]_-{\Rbb G\lambda} & \Rbb G\circ H_\Deu\circ F} } \quad \vcenter{ \xymatrix{\Lbb F\circ H_\Ceu\circ G \ar[r]^-{\Lbb F\rho} \ar[d]_{\lambda_G} & \Lbb F\circ\Rbb G\circ H_\Deu \ar[d]^{\dot\varepsilon_{H_\Deu}} \\ H_\Deu\circ F\circ G \ar[r]_-{H_\Deu\varepsilon} & H_\Deu & *!<2em,0mm>{.}} } \end{equation}
Even better, in the situation of the first theorem \ref{thm:adjoint gives derived functor}, these two diagrams again commute (by naturality of all arrows involved and the triangle identities). 

\begin{samepage}
\begin{definition}
  By a {\it derived adjunction\/}\index{Adjunction!derived}\index{Derived!adjunction} of~$(F \dashv G,\eta,\varepsilon)\colon \Ceu \rightleftarrows \Deu$ we mean an adjunction~$(\Lbb F \dashv \Rbb G,\dot\eta,\dot\varepsilon)\colon \Ho\Ceu \rightleftarrows \Ho\Deu$ together with transformations~$\lambda$ and~$\rho$ satisfying
  \begin{enumerate}
    \item $(\Lbb F,\lambda)$ is an absolute total left derived functor of~$F$;
    \item $(\Rbb G,\rho)$ is an absolute total right derived functor of~$G$;
    \item $\dot\eta$ and~$\dot\varepsilon$ are the unique natural transformations making the squares \ref{eqn:compatibility of units and counits} commute.
  \end{enumerate}
  If such a derived adjunction exists, we say that~$F\dashv G$ is {\it derivable\/}\index{Derivable!adjunction}\index{Adjunction!derivable}.
\end{definition}
\end{samepage}

With this definition, we can summarise the results of the two theorems above by the following (less precise) corollary.
\begin{corollary}
  Let~$F\dashv G$ be an adjunction of functors between categories with weak equivalences such that an absolute total right derived functor~$\Rbb G$ of~$G$ exists. Then~$F \dashv G$ is derivable if and only if\/~$\Rbb G$ has a left adjoint.\eop
\end{corollary}

  Also, using the theorems, we can easily study the question of when derived functors compose. Unfortunately, they do not in general but at least we can check it on adjoints.
\begin{corollary}
  Let $\Ceu$,~$\Deu$ and~$\Eeu$ be three categories with weak equivalences and
  \[ (F\dashv G,\eta,\varepsilon)\colon \Ceu \rightleftarrows \Deu,\qquad (F'\dashv G',\eta',\varepsilon')\colon \Deu \rightleftarrows \Eeu \]
  with derived adjunctions~$(\Lbb F\dashv\Rbb G,\dot\eta,\dot\varepsilon)$ and~$(\Lbb F'\dashv\Rbb G',\dot\eta',\dot\varepsilon')$. If~$(\Rbb G\circ\Rbb G',\rho'')$ is an absolute total right derived functor of~$G\circ G'$, then~$\Lbb F'\circ\Lbb F$ is an absolute total left derived functor of~$F'\circ F$ with counit
 \[ (\dot\varepsilon'\circ \Lbb F'\dot\varepsilon_{\Rbb G'})_{H_\Eeu F' F}\circ \Lbb F'\Lbb F\rho''_{F' F} \circ \Lbb F'\Lbb FH_\Ceu (G\dot\eta_F\circ\eta). \]
\end{corollary}
\begin{proof}
 Composing the two adjunctions as well as their derived adjunctions yields
  \[ \bigl(F\circ F \dashv G \circ  G', G\eta'_{F}\circ\eta, \varepsilon'\circ F'\varepsilon_{ G'}\bigr)\colon \Ceu \rightleftarrows \Eeu \qquad\text{and} \]
  \[ \bigl(\Lbb F'\circ \Lbb F \dashv \Rbb G \circ \Rbb G', \Rbb G\dot\eta'_{\Lbb F}\circ\dot\eta, \dot\varepsilon'\circ\Lbb F'\dot\varepsilon_{\Rbb  G'}\bigr)\colon \Ho\Ceu \rightleftarrows \Ho\Eeu. \]
  Now if~$(\Rbb G\circ\Rbb G',\rho')$ is an absolute total right derived functor of~$G\circ G'$, then by \ref{thm:adjoint gives derived functor} $\Lbb  F'\circ \Lbb F$ is indeed an absolute total left derived functor of~$ F'\circ F$ with counit
 \[ (\dot\varepsilon'\circ \Lbb F'\dot\varepsilon_{\Rbb G'})_{H_\Eeu F'F}\circ \Lbb F'\Lbb F\rho''_{F'F} \circ \Lbb F'\Lbb FH_\Ceu (G\eta'_F\circ\eta). \]
\end{proof}

As a special instance of this corollary, we can consider the case where~$\rho''$ is the canonical candidate~$\rho'' = \Rbb G\rho'\circ\rho_{G'}$ for a unit of~$\Rbb G\circ\Rbb G'$. It then follows that a counit of the total left derived functor~$\Lbb F'\circ\Lbb F$ is again given by the canonical candidate and vice versa. Note that this result is neither stronger nor weaker than the previous one.
\begin{corollary}\label{cor:composition of derived adjunctions formal}
  In the same situation as in the last corollary, let us write
  \[ (\Lbb F,\lambda),\quad (\Rbb G,\rho),\quad (\Lbb F',\lambda')\quad and\quad (\Rbb G',\rho') \]
  for the absolute total derived functors. Then the composite~$(\Lbb  F'\circ \Lbb F,\lambda'_F\circ\Lbb F'\lambda)$ is an absolute total left derived functor of~$ F'\circ F$ if and only if~$(\Rbb G\circ\Rbb  G',\Rbb G\rho'\circ\rho_{ G'})$ is an absolute total right derived functor of~$G\circ  G'$.
\end{corollary}
\begin{proof}
  If~$(\Rbb G\circ\Rbb  G', \Rbb G\rho'\circ\rho_{ G'})$ is an absolute total right derived functor, we apply the last corollary to $\rho'' \defas \Rbb G\rho'\circ\rho_{ G'}$ and the claim follows by a routine (albeit tedious) calculation:
 \begin{align*}
   \MoveEqLeft[2]\hskip1em\dot\varepsilon'_{H_\Eeu  F' F}\circ \Lbb F'\dot\varepsilon_{\Rbb G'H_\Eeu  F' F}\circ \Lbb  F'\Lbb F\Rbb G\rho'_{ F' F}\circ \Lbb F'\Lbb F\rho_{ G' F' F} \circ \Lbb F'\Lbb FH_\Ceu G\eta'_F \\
   &\hskip27em \circ \Lbb F'\Lbb FH_\Ceu\eta \\
   &\stackrel{\scriptscriptstyle\mathclap{\varepsilon\text{ nat}}}{=}\ \dot\varepsilon'_{H_\Eeu F'F} \circ \Lbb F'\rho_{ F'F} \circ \Lbb F'\dot\varepsilon_{H_\Deu G' F'F} \circ \Lbb F'\Lbb F\rho_{ G' F'F} \circ \Lbb F'\Lbb FH_\Ceu G\eta'_F\\
   &\hskip27em \circ \Lbb F'\Lbb FH_\Ceu\eta \\
   &\stackrel{\scriptscriptstyle\mathclap{\text{\ref{eqn:compatibility of units and counits}}}}{=}\ H_\Eeu\varepsilon'_{ F'F} \circ \lambda'_{ G' F'F} \circ \Lbb F'H_\Deu\varepsilon_{ G' F'F} \circ \Lbb F'\lambda_{G G' F'F} \circ \Lbb F'\Lbb FH_\Ceu G\eta'_F\\
   &\hskip27em\circ \Lbb F'\Lbb FH_\Ceu\eta \\
   &\stackrel{\scriptscriptstyle\mathclap{\lambda\text{ nat}}}{=}\ H_\Eeu\varepsilon'_{ F'F} \circ \lambda'_{ G' F'F} \circ \Lbb F'H_\Deu\varepsilon_{ G' F'F} \circ \Lbb F'H_\Deu F G\eta'_F \circ \Lbb F'H_\Deu F\eta \circ \Lbb F'\lambda \\
   &\stackrel{\scriptscriptstyle\mathclap{\lambda'\text{ nat}}}{=}\ H_\Eeu\varepsilon'_{ F'F} \circ H_\Eeu F'\varepsilon_{ G' F'F} \circ H_\Eeu F' F G\eta'_F \circ H_\Eeu F'F\eta \circ \lambda'_F \circ \Lbb F'\lambda \\
   &\stackrel{\scriptscriptstyle\mathclap{\varepsilon\text{ nat}}}{=}\ H_\Eeu\varepsilon'_{ F'F} \circ H_\Eeu F'\ol\eta_F \circ H_\Eeu F'\varepsilon_{F} \circ H_\Eeu F'F\eta \circ \lambda'_F \circ \Lbb F'\lambda \stackrel{\text{\tiny $\Delta$-id}}{=} \lambda'_F \circ \Lbb F'\lambda.
 \end{align*}
The other direction is dual.
\end{proof}

\begin{definition}
  Let $\Ceu$,~$\Deu$ and~$\Eeu$ be categories equipped with weak equivalences and let $F\colon \Ceu \to \Deu$,~$F'\colon \Deu \to \Eeu$ have absolute total left derived functors~$(\Lbb F,\lambda)$ and~$(\Lbb F',\lambda')$. By abuse of notation, we say that~$\Lbb F$ {\it composes\/}\index{Derived!functor!composing} with~$\Lbb F'$ iff~$(\Lbb F'\circ\Lbb F,\lambda'_F\circ \Lbb F'\lambda)$ is an absolute total left derived functor of~$F'\circ F$. Dually for right derived functors.
\end{definition}

\begin{remark}
  Clearly, this definition is equivalent to requiring that the composite~$F'\circ F$ have a total left derived functor $\bigl(\Lbb(F'\circ F),\lambda''\bigr)$ and that the natural transformation
  \[ \Lbb F' \circ \Lbb F \rar \Lbb(F'\circ F) \qquad\text{induced by $\lambda'_F\circ \Lbb F'\lambda$ be an isomorphism.} \]
\end{remark}

\begin{remark}\label{rem:composition of derived adjunctions}
  With the above definition, the last corollary \ref{cor:composition of derived adjunctions formal} can be restated by saying that for two derived adjunctions, the left adjoints compose iff the right adjoints compose.
\end{remark}

\begin{example}\label{ex:quillen functors compose}
  Obviously, if~$F'$ is homotopical, then~$\Lbb F' \cong \Ho F'$, so that~$\Lbb F$ and~$\Lbb F'$ compose. Also, total left derived functors of left Quillen functors between model categories compose. Dually for right Quillen functors.
\end{example}
\begin{example}
  More generally, let $\Ceu$,~$\Deu$ and~$\Eeu$ be equipped with weak equivalences and $H_\Ceu$, $H_\Deu$,~$H_\Eeu$ their respective localisations. If we are given
  \[ \Ceu \xrightarrow{F} \Deu \xrightarrow{F'} \Eeu \]
  together with a left $F$\hyph deformation retract~$\Ceu_0 \subseteq \Ceu$ and a left $F'$\hyph deformation retract~$\Deu_0 \subseteq \Deu$ such that~$F\Ceu_0 \subseteq \Deu_0$, then $\Lbb F'$ and~$\Lbb F$ compose.
\end{example}
\begin{proof}
  Let us write $\ol F \defas \ol{F\vert_{\Ceu_0}}$,~$\ol F' \defas \ol{F'\vert_{\Deu_0}}$ and fix some left deformation retractions $(\Ceu_0,Q,q)$, $(\Deu_0,Q',q')$, so that we may choose 
  \[ \Lbb F = \ol F\circ\widetilde Q,\ \Lbb F' = \ol F'\circ\widetilde Q' \qquad\text{with counits}\qquad \lambda \defas H_\Deu Fq,\ \lambda' \defas H_\Eeu F'q'. \]
  By hypothesis~$(\Ceu_0,Q,q)$ is a left $(F'F)$\hyph deformation retraction and~$\ol{F'F} \defas \ol{F'F\vert_{\Ceu_0}}$ comes with a natural isomorphism~$\omega\colon \Lbb F'\circ\Lbb F \cong \ol{F'F}\circ\widetilde Q$ determined by the components
  \[ \omega_{H_\Ceu C}\colon \Lbb F'\Lbb FH_\Ceu C = H_\Eeu F'Q'FQC \xrightarrow{H_\Eeu F'q'_{FQC}} H_\Eeu F'FQC \qquad\text{for }C \in \Ob\Ceu. \]
  Moreover, $\omega$ is compatible with the claimed counit of~$\Lbb F'\circ\Lbb F$ in the sense that
  \[ H_\Eeu F'Fq_C\circ\omega_{H_\Ceu C} = H_\Eeu F'(Fq_C\circ q'_{FQC}) = H_\Eeu F'(q'_{FC}\circ Q'Fq_C) = \lambda'_{FC}\circ \Lbb F'\lambda_C. \]
\end{proof}

\begin{example}
  As a special instance of the previous example, that is relevant to our main topic of the pointwiseness of homotopy colimits is the following. Let~$\Ceu$ is a category with weak equivalences, $\Ieu$,~$\Jeu$ index categories and $\Deu \subseteq \Ceu^\Ieu$ is a left $\colim$-deformation retract with a functorial replacement $Q\colon \Ceu^\Ieu \to \Deu$ (i.e.~a left deformation retract in the sense of \cite{DHKS}). Then~$\Lbb\colim_\Ieu \cong \Ho({\colim_\Ieu}\circ Q)$ and for every $J\in\Ob\Jeu$, the evaluation functor $\Ho\ev_J\colon \Ceu^{\Ieu\times\Jeu} \to \Ceu^\Ieu$ composes with $\Lbb\colim_\Ieu\colon \Ceu^\Ieu \to \Ceu$. To wit, as observed in \ref{rem:functorial deformations}, we can just apply~$Q$ pointwise, yielding a replacement functor $Q^\Jeu\colon \Ceu^{\Ieu\times\Jeu} \to \Deu^\Jeu$. With this
  \[ {\Lbb\colim_\Ieu}\circ\Ho\ev_J \cong \Ho({\colim_\Ieu}\circ Q\circ\ev_J) = \Ho({\colim_\Ieu}\circ\ev_J\circ Q^\Jeu) \cong \Lbb({\colim_\Ieu}\circ\ev_J). \]
\end{example}

Finally, we will quickly go one dimension higher and study the connection between mates and derivable adjunctions. For this, we consider two derivable adjunctions $F \dashv G$,~$F' \dashv G'$ (as always writing $\lambda$, $\lambda'$,~$\rho$ and~$\rho'$ for the universal morphisms of $\Lbb F$, $\Lbb F'$,~$\Rbb G$ and~$\Rbb G'$) together with two homotopical functors~$X$ and~$Y$ as in the following diagram:
\[ \vcenter{\xymatrix{ \Ceu \ar@<1ex>[r]^-{F}^-{}="1" \ar[d]_{X} & \Deu \ar@<1ex>[l]^-{G}^-{}="2" \ar[d]^{Y} \\ \Ceu' \ar@<1ex>[r]^-{F'}^-{}="11" & \ar@<1ex>[l]^-{G'}^-{}="12" \Deu' \ar@{}"1";"2"|(.3){\hbox{}}="3"|(.7){\hbox{}}="4" \ar@{|-}"4";"3" \ar@{}"11";"12"|(.3){\hbox{}}="13"|(.7){\hbox{}}="14" \ar@{|-}"14";"13" }}
\qquad\text{yielding}\qquad
\vcenter{\xymatrix{ \Ho\Ceu \ar@<1ex>[r]^-{\Lbb F}^-{}="1" \ar[d]_{\Ho X} & \Ho\Deu \ar@<1ex>[l]^-{\Rbb G}^-{}="2" \ar[d]^{\Ho Y} \\ \Ho\Ceu' \ar@<1ex>[r]^-{\Lbb F'}^-{}="11" & \ar@<1ex>[l]^-{\Rbb G'}^-{}="12" \Ho\Deu' & *!<2em,1ex>{.} \ar@{}"1";"2"|(.3){\hbox{}}="3"|(.7){\hbox{}}="4" \ar@{|-}"4";"3" \ar@{}"11";"12"|(.3){\hbox{}}="13"|(.7){\hbox{}}="14" \ar@{|-}"14";"13" }} \]
In addition, we require~$\Ho X$ to compose with~$\Lbb F'$ and~$\Ho Y$ to compose with~$\Rbb G'$ (all other reasonable combinations automatically compose as remarked in \ref{ex:quillen functors compose}).
\begin{theorem}\label{thm:derived mates}
  In the above situation, if~$\sigma\colon F'X \rar YF$ and~$\tau\colon XG \rar G'Y$ are mates, then the induced transformations between the derived functors
  \[ \Lbb\sigma\colon \Lbb F'\circ \Ho X \rar \Ho Y\circ\Lbb F \qquad\text{and}\qquad \Rbb\tau\colon \Ho X\circ \Rbb G \rar \Rbb G'\circ\Ho Y \]
  are mates, too (note that for this to even make sense, we need that~$\Ho X$ composes with~$\Lbb F'$ and that~$\Ho Y$ composes with~$\Rbb G'$). In particular, for $\Ceu = \Ceu'$,~$\Deu = \Deu'$ and $X$,~$Y$ identities, if~$\sigma$ and~$\tau$ are conjugate, so are~$\Lbb\sigma$ and~$\Rbb\tau$.
\end{theorem}
\begin{proof}
  First recall that the composite counits of~$\Lbb F'\circ\Ho X$ and~$\Ho Y\circ\Lbb F$ are respectively~$\lambda'_{X}$ and~$(\Ho Y)\lambda$ while the composite units of~$\Ho X\circ\Rbb G$ and~$\Rbb G'\circ\Ho Y$ are~$(\Ho X)\rho$ and~$\rho'_{Y}$. It follows that~$\Lbb\sigma$ and~$\Rbb\tau$ are defined by the equations
  \[ (\Ho Y)\lambda\circ(\Lbb\sigma)_{H_\Ceu} = H_{\Deu'}\sigma\circ\lambda'_X \qquad\text{and}\qquad (\Rbb\tau)_{H_\Deu} \circ (\Ho X)\rho = \rho'_{Y} \circ H_{\Ceu'}\tau. \]
  We now need to check that~$(\Lbb\sigma)^\sharp = \Rbb G' \Lbb\sigma\circ\dot\eta'_{\Ho X} = (\Rbb\tau)_{\Lbb F}\circ X\dot\eta$, which are natural transformations~$\Ho X \rar \Rbb G'\circ\Ho Y\circ \Lbb F$. By the universal property of a localisation and~$\Lbb F$ being an absolute total left derived functor, it suffices to check this after precomposition with~$H_\Ceu$ and postcomposition with~$(\Rbb G'\circ\Ho Y)\lambda$, where then
  \begin{align*}
    \MoveEqLeft (\Rbb G'\circ\Ho Y)\lambda\circ (\Rbb G'\Lbb\sigma\circ\dot\eta'_{\Ho X})_{H_\Ceu} = \Rbb G'\bigl((\Ho Y)\lambda\circ\Lbb\sigma_{H_\Ceu})\circ\dot\eta'_{H_{\Ceu'}X} \\
    &= \Rbb G'(H_{\Deu'}\sigma\circ\lambda'_X)\circ\dot\eta'_{H_{\Ceu'}X} = \Rbb G'H_{\Deu'}\sigma\circ(\Rbb G'\lambda'\circ\dot\eta'_{H_{\Ceu'}})_X \\
    &= \Rbb G'H_{\Deu'}\sigma\circ(\rho'_{F'}\circ H_{\Ceu'}\eta')_X = \Rbb G'H_{\Deu'}\sigma\circ\rho'_{F'X}\circ H_{\Ceu'}\eta'_X \\
    &= \rho'_{YF}\circ H_{\Ceu'}G'\sigma\circ H_{\Ceu'}\eta'_X = \rho'_{YF}\circ H_{\Ceu'}(G'\sigma\circ\eta'_X) \\
    &= \rho'_{YF}\circ H_{\Ceu'}(\tau_F\circ X\eta) = (\rho'_Y\circ H_{\Ceu'}\tau)_F\circ H_{\Ceu'}X\eta \\
    &= (\Rbb\tau)_{H_\Deu F}\circ (\Ho X)\rho_F \circ (\Ho X)H_\Ceu\eta = (\Rbb\tau)_{H_\Deu F}\circ(\Ho X)(\rho_F\circ H_\Ceu\eta) \\
    &= (\Rbb\tau)_{H_\Deu F}\circ (\Ho X)(\Rbb G\lambda\circ\dot\eta_{H_\Ceu}) = (\Rbb\tau)_{H_\Deu F}\circ(\Ho X\circ\Rbb G)\lambda\circ(\Ho X)\dot\eta_{H_\Ceu} \\
    &= (\Rbb G'\circ\Ho Y)\lambda\circ(\Rbb\tau)_{\Lbb F H_\Ceu}\circ(\Ho X)\dot\eta_{H_\Ceu} \\
    &= \Rbb G'(\Ho Y)\lambda\circ\bigl((\Rbb\tau)_{\Lbb F}\circ(\Ho X)\dot\eta)_{H_\Ceu} \\
  \end{align*}
\end{proof}

\section{Two Notions of Homotopy Colimits}
As an application of our theorems, we are going to show that homotopy colimits in the sense of derivator theory are the same as homotopy colimits in the Quillen model category sense. To make this precise, let us fix the following convention for convenience.
\begin{convention}
  Whenever we have a category~$\Ceu$ equipped with a class of weak equivalences as well as some index category~$\Ieu$ we always equip the diagram category~$\Ceu^\Ieu$ with the class of pointwise weak equivalences. That is to say, a natural transformation~$\tau\colon X \rar Y$ of diagrams $X$,~$Y\colon \Ieu \to \Ceu$ is a weak equivalence iff~$\tau_I\colon X_I \to Y_I$ is a weak equivalence in~$\Ceu$ for all~$I \in \Ob\Ieu$.
\end{convention}

With this natural definition of weak equivalences in a diagram category, one has the problem that in general, colimit and limit functors do not preserve them. The classical example is given by the commutative diagram in~$\Top$
\[ \xymatrix@R=2em{ D^n \ar[d] & S^{n-1} \ar@{_c->}[l] \ar@{^c->}[r] \ar@{=}[d] & D^n \ar[d] \\ \ast & S^{n-1} \ar@{_c->}[l] \ar@{^c->}[r] & \ast & *!<7mm,0mm>{,} } \]
where all vertical arrows are (weak) homotopy equivalences but the pushout of the top row is~$S^n$, while the one of the bottom row is~$\ast$ and these spaces are not (weakly) homotopy equivalent. So, in general, we cannot form~$\Ho(\colim)$ and the best we can do is trying to take a derived functor.%
\begin{observation}
  The constant diagram functor~$\Delta\colon \Ceu \to \Ceu^\Ieu$ is homotopical and thus induces a unique~$\Ho\Delta\colon \Ho\Ceu \to \Ho(\Ceu^\Ieu)$ such that $\Ho\Delta\circ H_\Ceu = H_{\Ceu^\Ieu}\circ\Delta$ (where~$H_\Ceu$ and~$H_{\Ceu^\Ieu}$ are the corresponding localisations).
\end{observation}
\begin{definition}
  For~$\Ceu$ a category equipped with weak equivalences and~$\Ieu$ an index category, a {\it homotopy colimit functor\/}\index{Homotopy!colimit!Grothendieck}\index{Homotopy!colimit}\index{Colimit!homotopy} is a functor~$\hocolim_\Ieu\colon \Ho(\Ceu^\Ieu) \to \Ho\Ceu$ left adjoint to~$\Ho\Delta$. If~$\Ceu$ has $\Ieu$\hyph colimits (so that~$\Delta\colon \Ceu \to \Ceu^\Ieu$ has a left adjoint~$\colim$) a {\it derived colimit functor\/}\index{Homotopy!colimit!Quillen} is an absolute total left derived functor~$\Lbb\colim_\Ieu$ of~$\colim_\Ieu$.
\end{definition}
\begin{proposition}
  If~$\Ceu$ has $\Ieu$\hyph colimits, then a homotopy colimit functor exists iff a derived colimit functor exists and they are the same.
\end{proposition}
\begin{proof}
  The functor~$\Delta\colon \Ceu \to \Ceu^\Ieu$ is homotopical and so~$\Rbb\Delta = \Ho\Delta$ exists.
\end{proof}

Although, by this result the existence of homotopy colimits is weaker than the existence of derived colimits (one does not need the existence of strict colimits), we prefer the more restrictive notion for what follows.%
\begin{definition}
  We say that a category~$\Ceu$ with weak equivalences {\it has homotopy colimits of type~$\Ieu$\/} (or simply that it has {\it $\Ieu$\hyph homotopy colimits\/}) iff the adjunction~$\colim_\Ieu\colon \Ceu^\Ieu \rightleftarrows \Ceu\lon\Delta$ exists and is derivable. Also, when writing ``$\hocolim_\Ieu$'', we always mean ``$\Lbb\colim_\Ieu$''.
\end{definition}

\section{Evaluation and Endomorphisms}
To say that colimits in a diagram category~$\Ceu^\Jeu$ are calculated pointwise means that
\[ {\colim}\circ\ev_J = \ev_J\circ{\colim} \text{ (or rather that there is a natural isomorphism)} \]
for all $J \in \Ob\Jeu$, where~$\ev_J$ are evaluation functors. So in order to prove a similar statement for homotopy colimits, we should try to better understand these.
\begin{proposition}\label{prop:adjoints of evaluation}
  Let~$\Ceu$ be a category and~$\Jeu$ an index category. For~$J \in \Ob\Jeu$ with inclusion~$\Stdin_J\colon \{J\} \hookrightarrow \Jeu$, the evaluation functor~$\ev_J\colon \Ceu^\Jeu \to \Ceu$ has a right adjoint
  \[ J_*\colon \Ceu \to \Ceu^\Jeu \qquad\text{given by}\qquad J_*C = C^{\Jeu(-,J)} \]
  (with the obvious arrow map), granted all these powers in~$\Ceu$ exist. Moreover, if\/ $\End_\Jeu(J) = \{\id_J\}$, then $J_*$ is fully faithful or equivalently, a counit~$\varepsilon$ of the adjunction is invertible. Dually, $\ev_J$ has a left adjoint
\[ J_!\colon \Ceu \to \Ceu^\Jeu \qquad\text{given by}\qquad J_!C = \Jeu(J,-)\cdot C, \]
granted all these copowers exist and if\/ $\End_\Jeu(J) = \{\id_J\}$, then~$J_!$ is fully faithful.
\end{proposition}
\begin{proof}
  Note that~$\ev_J = J^*$ is given by precomposition with~$\Stdin_J\colon \{J\} \hookrightarrow \Jeu$ and thus has a right adjoint given by taking the pointwise right Kan extension along~$J$ (granted that all these exist)
  \[ C \mapsto J_*C = \lim(-\comma \Stdin_J \to \{C\} \hookrightarrow \Ceu), \]
  so that~$(J_*C)_{J'} = C^{\pi_0(J'\comma\Stdin_J)}$ is the limit of a constant diagram for~$J' \in \Ob\Jeu$. But $J'\comma\Stdin_J$ is a discrete category with objects~$\Jeu(J',J)$. The counit's component at $C \in \Ob\Ceu$ is $\pr_{\id_J}\colon C^{\End_\Jeu(J)} \to C$, which is invertible if~$\End_\Jeu(J) = \{\id_J\}$.
\end{proof}

\begin{example}
  For~$\Jeu$ any direct or inverse category (in the sense of the theory of Reedy categories), then~$J_!$ and~$J_*$ (assuming they exist) are fully faithful for all $J\in\Ob\Jeu$. In particular, this is true for the indexing categories of the most common homotopy colimits, such as homotopy pushouts, coproducts, telescopes and coequalisers.
\end{example}

\begin{observation}\label{obs:beck-chevalley and evaluation}
  Given a category~$\Ceu$, index categories $\Ieu$,~$\Jeu$ and an object~$J \in \Ob\Jeu$ such that $J_*\colon \Ceu \to \Ceu^\Jeu$ as above exists, then so does $J_*\colon \Ceu^\Ieu \to (\Ceu^\Ieu)^\Jeu \cong \Ceu^{\Ieu\times\Jeu}$ (all powers are calculated pointwise) and an easy calculation shows that
  \[ \xymatrix{ \Ceu^\Jeu \ar[r]^-{\ev_J} \ar[d]_-{\Delta} & \Ceu \ar[d]^-{\Delta} \\ \Ceu^{\Ieu\times\Jeu} \ar[r]_{\ev_J} & \Ceu^\Ieu } \]
  (filled with the identity) satisfies the dual Beck-Chevalley condition. Indeed, the identity's mate is again the identity, for which we only need to check that~$\Delta\varepsilon = \varepsilon'_\Delta$, where~$\varepsilon$ is the counit of the top adjunction and~$\varepsilon'$ the counit of the bottom one. So let $C \in \Ob\Ceu$, $I \in \Ob\Ieu$ and then
  \[ (\varepsilon'_{\Delta C})_{I} = (\pr_{\id_J})_I\colon \bigl((\Delta C)^{\Jeu(J,J)}\bigr)_I = \bigl((\Delta C)I\bigr)^{\Jeu(J,J)} = C^{\Jeu(J,J)} \xrightarrow{\pr_{\id_J}} C = (\Delta C)_I \]
  while $(\Delta\varepsilon_C)_I = \varepsilon_C = \pr_{\id_J}$. If a right adjoint~$J_*$ to~$\ev_J$ exists but is not calculated pointwise as in the proposition, then the square still satisfies the dual Beck-Chevalley condition, granted that~$\Ceu$ has $\Jeu$\hyph colimits. This follows from Beck-Chevalley interchange and \ref{ex:beck-chevalley for colimits}.
\end{observation}

\begin{corollary}\label{cor:homotopical adjoints to evaluations}
  If~$\Ceu$ is a category with weak equivalences that has a terminal object, $\Jeu$ a preorder and~$J\in\Ob\Jeu$, then~$J_*\colon \Ceu \to \Ceu^\Jeu$ exists and is homotopical. For a general indexing category~$\Jeu$ (not necessarily a preorder), if~$\Ceu$ has all powers and weak equivalences are stable under them, then all functors $J_*\colon \Ceu \to \Ceu^\Jeu$ exist and are homotopical.
\end{corollary}
\begin{proof}
  The second part follows directly from the formula for~$J_*$ given in the above proposition. For the first part, note that in a preorder, every~$\Jeu(J,J')$ is either empty or a singleton, so that all components of a $J_*w\colon J_*C \to J_*D$ induced by a weak equivalence~$w\colon C \to D$ in~$\Ceu$ are either~$w$ itself or the identity for the terminal object of~$\Ceu$.
\end{proof}

  In case~$\Ceu$ is a model category (in particular bicomplete) and~$\Ceu^\Jeu$ carries the projective model structure, we even obtain two Quillen adjunctions. The less trivial part of this result can also be found e.g.~in \cite[Lemme 3.1.12]{Cisinski2006}.
\begin{corollary}
  If~$\Ceu$ is a model category and~$\Jeu$ an index category such that the projective model structure on~$\Ceu^\Jeu$ exists, then~$J_!\dashv \ev_J\dashv J_*$ are both Quillen adjunctions for all~$J \in \Ob\Jeu$.
\end{corollary}
\begin{proof}
  The adjunction~$J_!\dashv\ev_J$ is easy because~$\ev_J$ preserves fibrations and acyclic fibrations. For the other adjunction, let~$f\colon C \to D$ be an (acyclic) fibration in~$\Ceu$ and~$J' \in \Ob\Jeu$. Then
  \[ J_*(C \xrightarrow{f} D)_{J'} = C^{\Jeu(J',J)} \xrightarrow{f^{\Jeu(J',J)}} D^{\Jeu(J',J)} \]
  and a product of (acyclic) fibrations is again an (acyclic) fibration.
\end{proof}

\begin{example}\label{ex:existence of adjoints to ev for preorders}
  Let~$\Ceu$ be a category with weak equivalences that has a terminal object. If~$\Jeu$ is a small preorder, then~$\ev_J$ has a fully faithful right adjoint~$J_*$ and the adjunction
  \[ \adj[\ev_J][J_*]{\Ceu^\Jeu}{\Ceu} \qquad\text{lifts to}\qquad \adj[\Ho\ev_J][\Ho J_*][2.8em]{\Ho(\Ceu^\Jeu)}{\Ho\Ceu} \]
  for all~$J \in \Ob\Jeu$. Dually, if~$\Ceu$ has an initial object and~$\Jeu$ is a preorder, then~$\ev_J$ has a fully faithful left adjoint~$J_!$ and the adjunction
  \[ \radj[J_!][\ev_J]{\Ceu^\Jeu}{\Ceu} \qquad\text{lifts to}\qquad \radj[\Ho J_!][\Ho\ev_J][2.8em]{\Ho(\Ceu^\Jeu)}{\Ho\Ceu}. \]
\end{example}

\section{Derived Functors on Diagram Categories}
A common trick in diagrammatic homotopy theory is to first draw a $3\times 3$\hyph diagram of the form
\[ \xymatrix@R=1.4em@C=1.4em{ \bullet & \bullet \ar[l] \ar[r] & \bullet \\ \bullet \ar[u] \ar[d] & \bullet \ar[l]\ar[r]\ar[u]\ar[d] & \bullet \ar[u]\ar[d] \\ \bullet & \bullet \ar[l]\ar[r] & \bullet } \]
and then use that the homotopy colimit of this diagram can be calculated as a double homotopy pushout in two ways: rows first or columns first. This is usually taken for granted and under suitably nice hypotheses (e.g.~working in a cofibrantly generated model category) very easy to derive from the corresponding result for strict colimits.

When analysing this technique in order to weaken the hypotheses as far as possible, one sees that the whole trick is a combination of two results. First one uses a Fubini type result, which says that a homotopy colimit indexed by~$\Ieu \times \Jeu$ can be calculated as a double homotopy colimit and second, one uses that homotopy colimits in a diagram category~$\Ceu^\Jeu$ can be calculated pointwise (cf.~\ref{ex:beck-chevalley for colimits}). More specifically, the {\it Fubini theorem\/}\index{Fubini!theorem} is the following.
\begin{theorem}
  Let~$\Ceu$ be a category with weak equivalences and $\Ieu$,~$\Jeu$ index categories such that~$\Ceu^\Ieu$ has $\Jeu$\hyph homotopy colimits while~$\Ceu$ has $\Ieu$\hyph homotopy colimits. Then~$\Ceu$ has $\Ieu\times\Jeu$\hyph homotopy colimits and they are given by
  \[ \Ho(\Ceu^{\Ieu\times\Jeu}) \cong \Ho\bigl((\Ceu^\Ieu)^\Jeu\bigr) \xrightarrow{\hocolim_\Jeu} \Ho(\Ceu^\Ieu) \xrightarrow{\hocolim_\Ieu} \Ho\Ceu. \]
\end{theorem}
\begin{proof}
  Look at the right adjoints.
\end{proof}

For the second result, recall that if~$\Ceu$ is a cocomplete category, then every diagram category~$\Ceu^\Jeu$ is again cocomplete and the colimits are calculated pointwise. Unfortunately, for homotopy colimits, the matter is not that simple. For instance, there is no reason that the existence of homotopy colimits in a category~$\Ceu$ (equipped with a class of weak equivalences) should imply the existence of homotopy colimits in a diagram category~$\Ceu^\Jeu$ (with pointwise weak equivalences). However if we assume their existence, we can study interactions between homotopy colimits in different categories.

\begin{proposition}\label{prop:transfer of hocolims}
  Let~$\Ceu$ be a category with weak equivalences, $\Ieu$,~$\Jeu$ two index categories and~$J \in \Ob\Jeu$ such that~$J_! \dashv \ev_J \dashv J_*$ exist and are both derivable while $\End_\Jeu(J) = \{\id_J\}$. If~$\Ceu^\Jeu$ has $\Ieu$\hyph homotopy colimits, then~$\Ceu$ has $\Ieu$\hyph homotopy colimits, too.
\end{proposition}
\begin{proof}
  We obtain a left adjoint to~$\Ho\Delta\colon \Ho\Ceu \to \Ho(\Ceu^\Ieu)$ by
  \[ \xymatrix@1@C=3.5em{ *++{\Ho(\Ceu^\Ieu)} \ar@<1ex>[r]^-{\Lbb J_!}^-{}="1" & *++{\Ho\bigl((\Ceu^\Ieu)^\Jeu\bigr) \cong \Ho(\Ceu^{\Ieu\times\Jeu}) \cong \Ho\bigl((\Ceu^\Jeu)^\Ieu\bigr)} \ar@<1ex>[l]^-{\Ho\ev_J}^-{}="2" \ar@<1ex>[r]^-{\hocolim_\Ieu}^-{}="11" & *++{\Ho(\Ceu^\Jeu)} \ar@<1ex>[l]^-{\Ho\Delta}^-{}="12" \ar@<1ex>[r]^-{\Ho\ev_J}^-{}="21" & *++{\Ho\Ceu} \ar@<1ex>[l]^-{\Rbb J_*}^-{}="22"
  \ar@{}"1";"2"|(.3){\hbox{}}="3" \ar@{}"1";"2"|(.7){\hbox{}}="4" \ar@{|-}"4";"3"
  \ar@{}"11";"12"|(.3){\hbox{}}="13" \ar@{}"11";"12"|(.7){\hbox{}}="14" \ar@{|-}"14";"13"
  \ar@{}"21";"22"|(.3){\hbox{}}="23" \ar@{}"21";"22"|(.7){\hbox{}}="24" \ar@{|-}"24";"23" } \]
  because~$\Rbb J_*$ and~$\Ho(\ev_J\circ\Delta)$ compose (cf.\ \ref{ex:quillen functors compose}) and~$\ev_J\circ\Delta\circ J_* \cong \Delta$, where we need that $\End_\Jeu(J) = \{\id_J\}$ for the last isomorphism.
\end{proof}

\begin{example}
  The conditions of this proposition are satisfied if~$\Ceu$ has an initial and a terminal object while~$J$ is such that~$\Jeu(J',J)$ and $\Jeu(J,J')$ contain at most one element for all~$J' \in \Ob\Jeu$ (e.g.~$\Jeu$ a preorder).
\end{example}

\begin{remark}\label{rem:hocolims in diagram cats}
  If~$\Meu$ is a model category and~$\Ieu$ an arbitrary indexing category, then~$\Meu^\Ieu$ need not be a model category. However, it has a left model approximation in the sense of~\cite{HToD} and we can therefore still construct homotopy colimits in~$\Meu^\Ieu$.
\end{remark}

  As this proof illustrates, a basic tactic for establishing results about~$\hocolim_\Ieu$ is to switch to adjoints and work with~$\Ho\Delta$ instead. With this in mind, a direct attack on the ``pointwiseness'' of homotopy colimits in a diagram category might go as follows: If~$\Ho\Delta$ composes with~$\Rbb J_*$, then two composite right adjoints
  \[ \Rbb J_* \circ \Ho\Delta = \Rbb(J_*\circ\Delta) = \Rbb(\Delta\circ J_*) = \Ho\Delta\circ\Rbb J_* \]
  are the same and so its conjugate is an isomorphism between composite left adjoints
  \[ \Ho\ev_J\circ{\hocolim_\Ieu} \cong {\hocolim_\Ieu}\circ\Ho\ev_J. \]

  Although already a first step, this line of thought only tells us that we can calculate homotopy colimits pointwise on the object level and we cannot conclude anything on the level of arrows. More specifically, if~$j\colon J \to J'$ is an arrow, we obtain~$\ev_j\colon \Ceu^\Jeu \to \Ceu^{[1]}$, which induces~$\Ho\ev_j\colon \Ho(\Ceu^\Jeu) \to \Ho(\Ceu^{[1]})$, so that for~$X \in \Ob{\Ho(\Ceu^{\Ieu\times\Jeu})}$, 
  \[ (\Ho\ev_j)\hocolim_\Ieu X \in \Ob{\Ho(\Ceu^{[1]})}. \]
  On the other hand, if we view~$X(-,j)\colon X(-,J) \to X(-,J')$ as a morphism in~$\Ho(\Ceu^\Ieu)$ (i.e.~an object in~$\Ho(\Ceu^\Ieu)^{[1]}$), then applying~$\hocolim_\Ieu$ gives a morphism in~$\Ho\Ceu$ (i.e.~an object in~$(\Ho\Ceu)^{[1]}$) and we need to show that this is the same as
  \[ \ol{(H_\Ceu)_*}(\Ho\ev_j)\hocolim_\Ieu X, \quad\text{where }\ol{(H_\Ceu)_*}\text{ is defined by }\quad \vcenter{\xymatrix{ \Ceu^{[1]} \ar[r]^-{H_{\Ceu^{[1]}}} \ar[rd]_-{(H_\Ceu)_*} & \Ho(\Ceu^{[1]}) \ar[d]^{\ol{(H_\Ceu)_*}} \\ & (\Ho\Ceu)^{[1]} }} \]
  being commutative. In fact, it is not even clear why~$\hocolim_\Ieu X(-,j)$ should lie in the image of~$H_\Ceu$ (up to conjugation with a natural isomorphism). A more informal way to summarise this is that we would like to have an isomorphism ${\hocolim_\Ieu}\circ\Ho\ev_J \cong \Ho\ev_J\circ{\hocolim_\Ieu}$ that is natural in~$J \in \Ob\Jeu$; i.e.~such that for all~$j\colon J \to J'$ in~$\Jeu$ the following square in~$\Ho\Ceu$ is commutative:
  \[ \xymatrix{ \hocolim_\Ieu X(-,J) \ar[d]_{\hocolim_\Ieu X(-,j)} \ar@{}[r]|-*{\cong} & (\hocolim_\Ieu X)_J \ar[d]^{(\hocolim_\Ieu X)_j} \\ \hocolim_\Ieu X(-,J') \ar@{}[r]|-{\cong} & (\hocolim_\Ieu X)_{J'} & *!<1em,1ex>{.} } \]
  Again inspired by the corresponding strict result in the form of \ref{ex:beck-chevalley for colimits}, the (horizontal) Beck-Chevalley condition comes to our rescue.
\begin{theorem}\label{prop:hocolim and ev J}
  Let~$\Ceu$ be a category with weak equivalences, $\Ieu$,~$\Jeu$ two index categories and~$J \in \Ob\Jeu$ such that all derived adjunctions in the left-hand square
  \[ \vcenter{\xymatrix@C=3em{ \Ho\bigl((\Ceu^\Ieu)^\Jeu\bigr) \cong \Ho\bigl((\Ceu^\Jeu)^\Ieu\bigr) \ar@<-2.93em>[d]_{\Ho\ev_J}_{}="21" \ar@<1ex>[r]^-{\hocolim_\Ieu}^-{}="1" & \Ho(\Ceu^\Jeu) \ar@<-1ex>[d]_{\Ho\ev_J}_{}="31" \ar@<1ex>[l]^-{\Ho\Delta}^-{}="2" \\ 
*+!<2.5em,0mm>{\Ho(\Ceu^\Ieu)} \ar@<2.07em>[u]_{\Rbb J_*}^{}="22" \ar@<1ex>[r]^-{\hocolim_\Ieu}^-{}="11" & \Ho\Ceu \ar@<1ex>[l]^-{\Ho\Delta}^-{}="12" \ar@<-1ex>[u]_{\Rbb J_*}^{}="32"
  \ar@{}"1";"2"|(.3){\hbox{}}="3" \ar@{}"1";"2"|(.7){\hbox{}}="4" \ar@{|-} "4";"3"
  \ar@{}"11";"12"|(.3){\hbox{}}="13" \ar@{}"11";"12"|(.7){\hbox{}}="14" \ar@{|-} "14";"13"
  \ar@{}"21";"22"|(.3){\hbox{}}="23" \ar@{}"21";"22"|(.7){\hbox{}}="24" \ar@{|-} "24";"23"
  \ar@{}"31";"32"|(.3){\hbox{}}="33" \ar@{}"31";"32"|(.7){\hbox{}}="34" \ar@{|-} "34";"33"
  }} \ \vcenter{\xymatrix{ \Ho\bigl((\Ceu^\Ieu)^\Jeu\bigr) \cong \Ho\bigl((\Ceu^\Jeu)^\Ieu\bigr) \ar@<-2.93em>[d]^{\Ho\ev_J} & \Ho(\Ceu^\Jeu) \ar[d]_{\Ho\ev_J}\ar[l]_-{\Ho\Delta} \\ *+!<2.5em,0mm>{\Ho(\Ceu^\Ieu)} & \Ho\Ceu \ar[l]^-{\Ho\Delta}}} \]
  exist and\/~$\Ho\Delta$ composes with\/~$\Rbb J_*$. Then the right-hand square (filled with the identity) satisfies the (horizontal) Beck-Chevalley condition. In particular, there is an isomorphism
  \[ \sigma\colon {\hocolim_\Ieu}\circ\Ho\ev_J \cong \Ho\ev_J\circ{\hocolim_\Ieu}, \qquad\text{natural in }J \in \Ob\Jeu. \]
\end{theorem}
\begin{proof}
  By Beck-Chevalley interchange, it suffices to check that the right-hand square satisfies the vertical dual Beck-Chevalley condition. For this, we observe that~$\Ho\ev_J$ composes with~$\hocolim_\Ieu$ since~$\Ho\Delta$ composes with~$\Rbb J_*$ (cf.\ \ref{rem:composition of derived adjunctions}) and that on the strict level, the right-hand square from the proposition satisfies the Beck-Chevalley condition (cf.\ \ref{obs:beck-chevalley and evaluation}). Now use \ref{thm:derived mates}. Finally, the ``in particular''-part is just an instance of \ref{rem:naturality of the mating bijection}. To wit, for $j\colon J \to J'$ in~$\Jeu$, the left-hand square below commutes iff the right-hand square does, where the two horizontal pairs are mates:
  \[ \xymatrix@C=1.5em{ \hocolim_\Ieu\circ\Ho\ev_J \ar[r]^-{\sigma} \ar[d]_{{\hocolim_\Ieu}\Ho\ev_j} & \Ho\ev_J\circ\hocolim_\Ieu \ar[d]_{(\Ho\ev_j)_{\hocolim_\Ieu}} \\ \hocolim_\Ieu\circ\Ho\ev_{J'} \ar[r]_-{\sigma'} & \Ho\ev_{J'}\circ\hocolim_\Ieu }
  \xymatrix@C=1.5em{ \Ho\ev_J\circ\Ho\Delta \ar[r]^-{\id} \ar[d]_{\mathllap{(\Ho\ev_j)_{\Ho\Delta}}} & \Ho\Delta\circ\Ho\ev_j \ar[d]_{(\Ho\Delta)\Ho\ev_j} \\ \Ho\ev_{J'}\circ\Ho\Delta \ar[r]_-{\id} & \Ho\Delta\circ\Ho\ev_{J'} & *!<1.8em,.5ex>{.} } \]
\end{proof}

\begin{corollary}\label{cor:special commutation of hocolims}
  Let $\Ceu$,~$\Ieu$,~$\Jeu$ be as in the theorem and assume that~$\Ceu^\Jeu$ has homotopy $\Ieu$\hyph colimits. If~$\Ceu$ has a terminal object and~$\Jeu$ is a preorder, there are isomorphisms
  \[ \sigma\colon {\hocolim_\Ieu}\circ\Ho\ev_J \cong \Ho\ev_J\circ{\hocolim_\Ieu}, \qquad\text{natural in }J \in \Ob\Jeu. \]
  In particular, we always have such natural isomorphisms for double homotopy pushouts, coproducts and telescopes.
\end{corollary}
\begin{proof}
  If~$\Jeu = \emptyset$ the statement is void, so we assume we can pick some $J \in \Ob\Jeu$ to construct homotopy $\Ieu$\hyph colimits in~$\Ceu$ as shown in~\ref{prop:transfer of hocolims}. By our hypotheses, all functors~$J_*$ exist and are homotopical, as observed in~\ref{cor:homotopical adjoints to evaluations}. But then~$\Ho\Delta$ and~$\Rbb J_* = \Ho J_*$ clearly compose and we can invoke the theorem.
\end{proof}

\begin{corollary}
   Let $\Ceu$,~$\Ieu$,~$\Jeu$ be as in the theorem and assume that both~$\Ceu$ and~$\Ceu^\Jeu$ have homotopy $\Ieu$\hyph colimits. If~$\Ceu$ has powers and its weak equivalences are stable under them, there are isomorphisms
  \[ \sigma\colon {\hocolim_\Ieu}\circ\Ho\ev_J \cong \Ho\ev_J\circ{\hocolim_\Ieu}, \qquad\text{natural in }J \in \Ob\Jeu. \]
\end{corollary}
\begin{proof}
  Again by~\ref{cor:homotopical adjoints to evaluations}, all functors~$J_*$ exist and are homotopical, so that $\Ho\Delta$ and~$\Rbb J_* = \Ho J_*$ compose.
\end{proof}

\begin{example}\label{ex:example1}
  The category~$\Ch_{\geq 0}(\Ab_{\rm fg})$ of non-negatively graded chain complexes of {\it finitely generated\/} abelian groups inherits a model structure from the usual model category of non-negatively graded chain complexes, where the weak equivalences are the quasi-isomorphisms, the fibrations are the levelwise surjections and the cofibrations the levelwise injections with projective (i.e.~free) cokernel. There are no obvious functorial cofibrant replacements within~$\Ch_{\geq 0}(\Ab_{\rm fg})$ and all the usual techniques fail since our base category is only finitely cocomplete, though the author was unable to find a proof in the literature that~$\Ch_{\geq 0}(\Ab_{\rm fg})$ is not cofibrantly generated.

Still, $\Ch_{\geq 0}(\Ab_{\rm fg})$ is a model category (albeit only a finitely bicomplete one) and we can construct finite homotopy colimits in it, as well as in all diagram categories~$\Ch_{\geq 0}(\Ab_{\rm fg})^\Ieu$ with~$\Ieu$ finite (see~\ref{rem:hocolims in diagram cats}). Since quasi-isomorphisms are stable under (finite) direct sums, the above corollary applies.
\end{example}

\begin{example}\label{ex:example2}
  As a more involved example, in~\cite{Isaksen2001} and~\cite{Isaksen2004}, Isaksen constructed model structures on the category of pro-simplicial sets and more general pro-objects in model categories. Being model categories, we can again construct homotopy colimits in them as well as all diagram categories. However, as shown in {\it op.~cit.}, the model structures in question are not cofibrantly generated and so, there are no projective model structures at our disposal. Still, our theorem applies and we at least always have \ref{cor:special commutation of hocolims} (since the stability of weak equivalences under powers depends on the model category we take pro-objects in).
\end{example}

Finally, let us mention an easy example that is not accessible to any of the previously mentioned techniques (model structures or left model approximations).
\begin{example}\label{ex:example3}
  Let~$\Deltabf$ be the usual indexing category for simplicial sets and $\Deltabf_+ \subset \Deltabf$ the wide subcategory of injections with inclusion $I\colon \Deltabf_+ \hookrightarrow \Deltabf$. The category of $\delta$\hyph {\it sets\/} (or {\it semi-simplicial sets\/}) is the presheaf category $\delta\hyph \Sets \defas \Sets^{\Deltabf_+}$. The left Kan extension~$I_!$ of a $\delta$\hyph set along~$I$ corresponds to freely adjoining degeneracies and we define a map of $\delta$\hyph sets to be a weak equivalence iff its image under~$I_!$ is one. Equivalently, a map of $\delta$\hyph sets is a weak equivalence iff it is one in~$\Top$ after geometric realisation.
  
  It was remarked by Peter May (who himself learned this from his student Matthew Thibault) that there cannot be a model structure with these weak equivalences. Still, we can construct homotopy colimits in~$\delta\hyph\Sets$ (or any category of diagrams in it) by passing to~$\sSets$ via~$I_!$, forming them there and then passing back to~$\delta\hyph\Sets$ via~$I^*$.
\end{example}

  In some situations (e.g.~when using left model approximations for the construction of homotopy colimits as in \cite{HToD}), it might not be possible to verify the hypotheses of the above theorem. Indeed, it might not even be possible to construct the right adjoint~$\Rbb J_*$. However, we can still use the left adjoint~$\Lbb J_!$ of~$\Ho\ev_J$ rather than~$\Rbb J_*$ to force naturality.
  \begin{proposition}\label{prop:hocolim and J shriek}
  Let~$\Ceu$ be a category with weak equivalences, $\Ieu$,~$\Jeu$ two index categories and~$J \in \Ob\Jeu$ such that all derived adjunctions in the square
  \[ \xymatrix@C=3.5em{ \Ho\bigl((\Ceu^\Ieu)^\Jeu\bigr)  \cong \Ho\bigl((\Ceu^\Jeu)^\Ieu\bigr) \ar@<-2.07em>[d]^{\Ho\ev_J}^{}="22" \ar@<1ex>[r]^-{\hocolim_\Ieu}^-{}="1" & \Ho(\Ceu^\Jeu) \ar@<1ex>[d]^{\Ho\ev_J}^{}="32" \ar@<1ex>[l]^-{\Ho\Delta}^-{}="2" \\ 
*+!<2.5em,0mm>{\Ho(\Ceu^\Ieu)} \ar@<2.93em>[u]^{\Lbb J_!}^{}="21" \ar@<1ex>[r]^-{\hocolim_\Ieu}^-{}="11" & \Ho\Ceu \ar@<1ex>[l]^-{\Ho\Delta}^-{}="12" \ar@<1ex>[u]^{\Lbb J_!}^{}="31"
  \ar@{}"1";"2"|(.3){\hbox{}}="3" \ar@{}"1";"2"|(.7){\hbox{}}="4" \ar@{|-} "4";"3"
  \ar@{}"11";"12"|(.3){\hbox{}}="13" \ar@{}"11";"12"|(.7){\hbox{}}="14" \ar@{|-} "14";"13"
  \ar@{}"21";"22"|(.3){\hbox{}}="23" \ar@{}"21";"22"|(.7){\hbox{}}="24" \ar@{|-} "24";"23"
  \ar@{}"31";"32"|(.3){\hbox{}}="33" \ar@{}"31";"32"|(.7){\hbox{}}="34" \ar@{|-} "34";"33"
  } \]
  exist and whose units we denote by $\eta$, $\theta$,~$\theta'$ and~$\eta'$ (top to bottom, left to right). Then there is
  \[ \alpha\colon {\hocolim_\Ieu}\circ\Lbb J_! \cong \Lbb J_!\circ{\hocolim_\Ieu} \]
  compatible with the composite units (and counits for that matter) in the sense that
  \[ \Ho(\ev_J\circ\Delta)\alpha\circ(\Ho\ev_J)\eta_{\Lbb J_!}\circ\theta = (\Ho\Delta)\theta'_{\hocolim_\Ieu}\circ\eta'. \]
\end{proposition}
\begin{proof}
  Obviously~$\ev_J\circ\Delta = \Delta\circ\ev_J$, whence the square of right adjoints commutes.
\end{proof}

  While the compatibility formula in this proposition might seem useless at first, note that if the~$\Lbb J_!$ are fully faithful (or equivalently, if $\theta$ and~$\theta'$ are isomorphisms), we get a good starting point for comparing~$\eta$ and~$\eta'$.
\begin{corollary}\label{cor:compatibility of hocolim units}
  Under the hypotheses of \ref{prop:hocolim and J shriek}, if the two~$\Lbb J_!$ are fully faithful and there is some isomorphism~$\hocolim_\Ieu\circ\Ho\ev_J \cong \Ho\ev_J\circ\hocolim_\Ieu$ (e.g.~if $\hocolim_\Ieu$ and\/~$\Ho\ev_J$ compose), then there is even an isomorphism
  \[ \beta_J\colon {\hocolim_\Ieu}\circ\Ho\ev_J \cong \Ho\ev_J\circ{\hocolim_\Ieu}\text{ such that }\,(\Ho\ev_J)\eta = (\Ho\Delta)\beta_J\circ \eta'_{\Ho\ev_J}. \]
\end{corollary}
\begin{proof}
  The compatibility condition of~$\alpha$ at its $(\Ho\ev_J)X$-component for a diagram~$X \in \Ob{\Ho\bigl((\Ceu^\Ieu)^\Jeu\bigr)}$ reads
  \[ (\Ho\Delta)\Bigl(\theta'^{-1}_{\hocolim_\Ieu}\!\circ(\Ho\ev_J)\alpha\Bigr)_{(\Ho\ev_J)X}\!\!\circ (\Ho\ev_J)\eta_{\Lbb J_!(\Ho\ev_J)X}\circ\theta_{(\Ho\ev_J)X} = \eta'_{(\Ho\ev_J)X} \]
  and we need to show that~$(\Ho\ev_J)\eta_{\Lbb J_!\ev_J X} = (\Ho\ev_J)\eta_X$ up to composition with suitable natural isomorphisms. For this, let us write~$\zeta$ for the counit of~$\Lbb J_!\dashv\Ho\ev_J$ and note that~$(\Ho\ev_J)\zeta$ is invertible with inverse~$\theta_{\Ho\ev_J}$. The naturality of~$\eta$ gives a commutative square
  \[ \xymatrix@C=7.5em{ (\Ho\ev_J)\Lbb J_!(\Ho\ev_J)X \ar[r]^-{(\Ho\ev_J)\eta_{\Lbb J_!(\Ho\ev_J)X}} \ar[d]_{(\Ho\ev_J)\zeta_X}^{\cong} & (\Ho\ev_J)(\Ho\Delta)\hocolim_\Ieu\Lbb J_!(\Ho\ev_J)X \ar[d]^{\cong}_{(\Ho\ev_J)(\Ho\Delta)\hocolim_\Ieu\zeta_X} \\ (\Ho\ev_J)X \ar[r]_-{(\Ho\ev_J)\eta_X} & (\Ho\ev_J)(\Ho\Delta)\hocolim_\Ieu X & *!<10em,.5ex>{,} } \]
  where the right-hand arrow $(\Ho\Delta)(\Ho\ev_J)\hocolim_\Ieu\zeta_X$ is invertible because~$\Ho\ev_J$ commutes with~$\hocolim_\Ieu$, while the left-hand arrow cancels with~$\theta_{(\Ho\ev_J)X}$ and so~$\eta'_{(\Ho\ev_J)X}$ equals
  \[ (\Ho\Delta)\Bigl(\theta'^{-1}_{\hocolim_\Ieu(\Ho\ev_J)}\circ(\Ho\ev_J)\alpha_{\Ho\ev_J} \circ \bigl((\Ho\ev_J)\hocolim_\Ieu\zeta\bigr)^{-1}\Bigr)_X \circ (\Ho\ev_J)\eta_X. \]
\end{proof}

  The last obstacle to overcome before assembling this all is that we need some criterion to decide when the~$\Lbb J_!$ are fully faithful. This turns out to be rather simple and it suffices if~$J_!$ is so (e.g.~as in \ref{prop:adjoints of evaluation}).%
\begin{proposition}\label{prop:fully faithful derived functor of J shriek}
  Let~$\Ceu$ be a category with weak equivalences, $\Jeu$ some index category and~$J \in \Ob\Jeu$ such that a fully faithful left adjoint~$J_!$ to~$\ev_J$ exists and~$J_!\dashv \ev_J$ is derivable. Then~$\Lbb J_!$ is fully faithful.
\end{proposition}
\begin{proof}
  Choosing compatible units and counits of the derived adjunction, \ref{eqn:compatibility of units and counits} gives us
  \[ \rho_F\circ H_\Ceu\eta = (\Ho\ev_J)\lambda\circ\dot\eta_{H_\Ceu}, \]
  which is an isomorphism because~$\Rbb\ev_J = \Ho\ev_J$ (i.e.~$\rho$ is invertible) and~$J_!$ is fully faithful (i.e.~$\eta$ is invertible). Now note that~$\Lbb J_!$ and~$\Ho\ev_J$ compose, whence $\Ho\ev_J\circ\Lbb J_! \cong \Lbb(\ev_J\circ J_!)$. But~$\ev_J\circ J_! \cong \id_\Ceu$, so that~$\Lbb(\ev_J\circ J_!) \cong \id_{\Ho\Ceu}$ and~$\lambda$ must be invertible. It follows that~$\dot\eta_{H_\Ceu}$ (or equivalently~$\dot\eta$) is invertible, thus proving our claim.
\end{proof}
\begin{remark}
  Obviously, this proposition generalises mutatis mutandis to an arbitrary derivable adjunction with a fully faithful left and a homotopical right adjoint.
\end{remark}

\begin{example}
 The proposition in particular applies in both of the following cases:
  \begin{enumerate}
    \item $\Ceu$ has an initial object and~$\Jeu$ is a preorder;
    \item $\Ceu$ is a model category such that the projective model structure on~$\Ceu^\Jeu$ exists and $\End_\Jeu(J) = \{\id_J\}$ (e.g.~$\Ceu$ any model category and~$\Jeu$ direct).
  \end{enumerate}
\end{example}

\begin{theorem}\label{thm:pointwiseness of hocolim}
  Let~$\Ceu$ be a category with weak equivalences and $\Ieu$,~$\Jeu$ two index categories such that for all~$J \in \Ob\Jeu$, the two evaluation functors~$\ev_J$ indicated below have fully faithful left adjoints~$J_!$ and the derived adjunctions in the square
  \[ \xymatrix@C=3.5em{ \Ho\bigl((\Ceu^\Ieu)^\Jeu\bigr)  \cong \Ho\bigl((\Ceu^\Jeu)^\Ieu\bigr) \ar@<-2.07em>[d]^{\Ho\ev_J}^{}="22" \ar@<1ex>[r]^-{\hocolim_\Ieu}^-{}="1" & \Ho(\Ceu^\Jeu) \ar@<1ex>[d]^{\Ho\ev_J}^{}="32" \ar@<1ex>[l]^-{\Ho\Delta}^-{}="2" \\ 
  *+!<2.5em,0mm>{\Ho(\Ceu^\Ieu)} \ar@{^{c}->}@<2.93em>[u]^{\Lbb J_!}^{}="21" \ar@<1ex>[r]^-{\hocolim_\Ieu}^-{}="11" & \Ho\Ceu \ar@<1ex>[l]^-{\Ho\Delta}^-{}="12" \ar@{^{c}->}@<1ex>[u]^{\Lbb J_!}^{}="31"
  \ar@{}"1";"2"|(.3){\hbox{}}="3" \ar@{}"1";"2"|(.7){\hbox{}}="4" \ar@{|-} "4";"3"
  \ar@{}"11";"12"|(.3){\hbox{}}="13" \ar@{}"11";"12"|(.7){\hbox{}}="14" \ar@{|-} "14";"13"
  \ar@{}"21";"22"|(.3){\hbox{}}="23" \ar@{}"21";"22"|(.7){\hbox{}}="24" \ar@{|-} "24";"23"
  \ar@{}"31";"32"|(.3){\hbox{}}="33" \ar@{}"31";"32"|(.7){\hbox{}}="34" \ar@{|-} "34";"33"
  } \]
  exist and there is an isomorphism~$\hocolim_\Ieu\circ\Ho\ev_J \cong \Ho\ev_J\circ\hocolim_\Ieu$. Then there is a family of isomorphisms
  \[ \beta_J\colon \hocolim_\Ieu\circ\Ho\ev_J \cong \Ho\ev_J\circ\hocolim_\Ieu \qquad\text{natural in~$J$.} \]
\end{theorem}
\begin{proof}
  For~$J \in \Ob\Jeu$ let us choose~$\beta_J$ as in \ref{cor:compatibility of hocolim units} and note that for~$j\colon J \to J'$ in~$\Jeu$ the arrow~$\hocolim_\Ieu X(-,j)$ is the unique~$f\colon \hocolim_\Ieu X(-,J) \to \hocolim_\Ieu X(-,J')$ in~$\Ho\Ceu$ making
  \[ \xymatrix@C=3em{ X(-,J) \ar[d]_{X(-,j)} \ar[r]^-{\eta'_{X(-,J)}} & (\Ho\Delta)\hocolim_\Ieu X(-,J) \ar@{..>}[d]^{(\Ho\Delta)f} \\ X(-,J') \ar[r]_-{\eta'_{X(-,J')}} & (\Ho\Delta)\hocolim_\Ieu X(-,J') } \]
  in~$\Ho(\Ceu^\Ieu)$ commute. Extending this square to the right using~$\beta_{J,X}$ and~$\beta_{J',X}$, the naturality of~$\Ho\ev_j$ yields a commutative rectangle of solid arrows in~$\Ho(\Ceu^\Ieu)$
  \[ \xymatrix@C=3.8em{ X(-,J) \ar[d]_{X(-,j)} \ar[r]^-{\eta'_{X(-,J)}} & (\Ho\Delta)\hocolim_\Ieu X(-,J) \ar@{..>}[d] \ar[r]^-{(\Ho\Delta)\beta_{J,X}}_{\cong} & (\Ho\Delta)(\hocolim_\Ieu X)J \ar[d]_{(\Ho\Delta)(\hocolim_\Ieu X)j} \\ X(-,J') \ar[r]_-{\eta'_{X(-,J')}} & (\Ho\Delta)\hocolim_\Ieu X(-,J) \ar[r]_-{(\Ho\Delta)\beta_{J',X}}^{\cong} & (\Ho\Delta)\hocolim_\Ieu X(-,J') } \]
  with the top and bottom composites being~$(\Ho\ev_J)\eta_X$ and~$(\Ho\ev_{J'})\eta_X$ respectively. Upon defining~$f \defas \beta_{J',X}^{-1}\circ(\hocolim_\Ieu X)j\circ\beta_{J,X}$ the dotted arrow~$(\Ho\Delta)f$ renders the entire diagram commutative and the claim follows.
\end{proof}

\bibliographystyle{amsplain}
\bibliography{homotopy.bib}

\end{document}